\documentclass[12pt,a4paper]{amsart}
\usepackage[a4paper,margin=30mm]{geometry}
\usepackage{amsmath,amssymb,amsthm,mathtools}
\allowdisplaybreaks
\usepackage{microtype}
\usepackage{enumitem}
\usepackage{needspace}
\usepackage[hidelinks]{hyperref}
\usepackage[nameinlink,noabbrev]{cleveref}
\hypersetup{
 pdftitle={Ledrappier--Young entropy formula for C1 diffeomorphisms with dominated splitting. Part 3: The full formula with arbitrary multiplicities},
 pdfauthor={Yao Tong},
 pdfsubject={Transverse entropy, conditional dimension, and dominated splitting},
 bookmarksnumbered=true
}
\newtheorem{theorem}{Theorem}[section]
\newtheorem{mainthm}[theorem]{Theorem}
\newtheorem{proposition}[theorem]{Proposition}
\newtheorem{lemma}[theorem]{Lemma}
\newtheorem{corollary}[theorem]{Corollary}
\theoremstyle{definition}
\newtheorem{definition}{Definition}
\theoremstyle{remark}
\newtheorem{remark}{Remark}
\crefname{mainthm}{Theorem}{Theorems}
\Crefname{mainthm}{Theorem}{Theorems}
\crefname{theorem}{Theorem}{Theorems}
\crefname{proposition}{Proposition}{Propositions}
\crefname{lemma}{Lemma}{Lemmas}
\crefname{corollary}{Corollary}{Corollaries}
\crefname{definition}{Definition}{Definitions}
\crefname{remark}{Remark}{Remarks}

\newcommand{\R}{\mathbb R}
\newcommand{\Z}{\mathbb Z}
\newcommand{\B}{\mathcal B}
\newcommand{\A}{\mathcal A}
\newcommand{\Q}{\mathcal Q}
\newcommand{\C}{\mathcal C}
\newcommand{\one}{\mathbf 1}
\newcommand{\eps}{\varepsilon}
\newcommand{\supp}{\operatorname{supp}}
\newcommand{\diam}{\operatorname{diam}}
\newcommand{\dist}{\operatorname{dist}}
\newcommand{\ld}{\underline d}
\newcommand{\ud}{\overline d}
\newcommand{\ldel}{\underline\delta}
\newcommand{\udel}{\overline\delta}
\newcommand{\Diff}{\operatorname{Diff}}
\newcommand{\tg}{\tilde g}
\newcommand{\tW}{\widetilde W}
\newcommand{\tp}{\tilde\pi}
\newcommand{\tm}{\tilde m}

\title[$C^1$ entropy formula, Part 3]
{Ledrappier--Young entropy formula for $C^1$ diffeomorphisms
with dominated splitting\\
Part 3: full formula
with arbitrary multiplicities}
\author[Yao Tong]{Yao Tong$^{1}$}
\thanks{$^{1}$\textsc{School of Mathematical Sciences, Peking University,
Beijing 100871, China.}\\
\emph{E-mail address:}\ \href{mailto:2401110024@stu.pku.edu.cn}{\texttt{2401110024@stu.pku.edu.cn}}.}
\date{}

\begin{document}
\begin{abstract}
We establish the full Ledrappier--Young entropy formula for
ergodic invariant measures of \(C^1\) diffeomorphisms when the positive
Oseledets bundles and their nonpositive complement form a
dominated splitting. We prove transverse formulas under three-block
domination and equality of metric and unstable entropy under
positive--nonpositive domination.
To complete the proof,
we develop a new mechanism for estimating partial entropy
through comparisons of conditional information
along orbits.

Applications characterize
vanishing transverse dimension, give upper pointwise dimension bounds
and exact dimensionality of hyperbolic measures. In the
quasi-isometric-center setting, this implies a \(C^1\) invariance
principle for measures with nonpositive center Lyapunov exponents.
\end{abstract}
\maketitle

\vspace*{-28pt}
\tableofcontents

\section{Introduction and main results}\label{sec:intro}

\subsection{Introduction}

  Metric entropy, Lyapunov exponents, and measure dimension
describe orbit complexity, tangent growth, and the scaling of
mass. The Margulis--Ruelle inequality bounds entropy by the
integral of the sum of the positive Lyapunov exponents, counted
with multiplicity \cite{Ruelle}. Pesin's entropy formula gives
equality for \(C^{1+\beta}\) diffeomorphisms preserving a smooth
volume, where \(\beta>0\) \cite{Pesin}; see also \cite{Mane,LS}.
Young \cite{Young} related entropy, Lyapunov exponents, and
measure dimension for \(C^{1+\beta}\) surface diffeomorphisms.

  Ledrappier and Young \cite{LY1,LY2} developed these relations
into two complementary theorems. Building on Ledrappier's work
on Sinai measures \cite{LedrappierSinai} and the work of Ledrappier and Strelcyn \cite{LS}, they proved that an
invariant probability measure of a \(C^2\) diffeomorphism of a
compact manifold satisfies Pesin's entropy formula if and only
if its unstable conditional measures are absolutely continuous
with respect to leaf volume \cite[Theorem~A]{LY1}. This gives
a geometric characterization of equality in Ruelle's inequality,
including for measures with zero Lyapunov exponents.

  Their second paper gives an entropy formula for arbitrary
ergodic invariant measures \cite{LY2}. The conditional measures
along the nested unstable manifolds associated with the positive
Lyapunov exponents are exact dimensional. Each increment of
partial entropy equals the corresponding
positive Lyapunov exponent multiplied by the increment of
conditional dimension. These dimension increments lie between
zero and the corresponding exponent multiplicities, and the
sum of their exponent-weighted contributions is the metric
entropy. The formula thus replaces the multiplicities in
Ruelle's upper bound with dimensions determined by the measure
itself, yielding an exact identity even for singular and
nonhyperbolic measures. Together, the two theorems explain
both how unstable geometry determines entropy and when the
measure realizes the full contribution of its positive
exponents.

  The classical proofs rely on Lipschitz holonomy estimates
to compare transverse scales and conditional measures.
Brown \cite{Brown} established the required regularity in
the relevant local models, extending the Ledrappier--Young
formulas to \(C^{1+\beta}\) diffeomorphisms. For dynamically
coherent \(C^{1+\beta}\) partially hyperbolic systems,
Saghin \cite{Saghin} proved differentiability of stable and
unstable holonomies between center leaves under the
corresponding \(\beta\)-bunching conditions.

  Extending the entropy--dimension formula to \(C^1\)
diffeomorphisms is a natural problem, but even its underlying
geometric objects need not exist. The examples of Pugh
\cite{Pugh} and Bonatti, Crovisier, and Shinohara \cite{BCS}
show that nonzero Lyapunov exponents alone do not guarantee
the invariant manifolds of Pesin theory. Dominated splitting
is therefore a natural hypothesis at this regularity: it
uniformly separates the relevant invariant directions and
allows one to recover nonuniform invariant manifolds, as
established by Abdenur, Bonatti, and Crovisier
\cite[Section~8]{ABC}.

  Several entropy results are available at \(C^1\) regularity.
Tahzibi \cite{Tahzibi} proved Pesin's formula for
\(C^1\)-generic volume-preserving surface diffeomorphisms.
Under domination, Sun and Tian \cite{SunTian} obtained entropy
lower bounds for absolutely continuous invariant measures,
while Catsigeras, Cerminara, and Enrich \cite{CCE} proved
related bounds for measures describing the asymptotic
statistics of Lebesgue-typical orbits. Wang, Wang, and Zhu
\cite{WWZ} established analogues of Ruelle's inequality and
Pesin's formula along unstable foliations.

  Recovering invariant manifolds does not, however, provide
the Lipschitz holonomy estimates used in the classical
Ledrappier--Young proof. The available \(C^1\) estimates are
H\"older in general, and the resulting distortion of transverse
scales prevents those comparisons from being applied directly.
The entropy bounds above do not resolve the exact relation
between partial entropy and conditional dimension for
arbitrary ergodic measures.

  Partial answers were obtained in our joint work \cite{GTY},
referred to as Part~1, and our subsequent paper \cite{Tong},
referred to as Part~2. Part~1 proves equality of metric and
unstable entropy under a dominated positive--zero--negative
splitting when the zero-exponent bundle has dimension at most
one. Part~2 establishes the transverse formula for a simple
intermediate positive exponent and the full formula when the
intermediate nonnegative exponents are simple. These arguments
work with H\"older holonomies, but use one-dimensional
uncentered differentiation.

The present paper removes the multiplicity restrictions and
the need to separate zero from negative exponents. At an
intermediate positive exponent, domination between its
Oseledets bundle, the sum of the faster bundles, and the
remaining bundles gives the transverse entropy formula.
Domination between the positive and nonpositive bundles gives
equality of unstable and metric entropy. Together, these
results establish the full Ledrappier--Young formula when the
positive Oseledets bundles and their nonpositive complement
form a dominated splitting. Thus we complete the Ledrappier--Young entropy formula theory in the classical $C^1$ plus domination settings.

Our main new ingredient is a mechanism for estimating partial
entropy through conditional-information comparisons.
The guiding principle, also present in the work of Tahzibi
and Yang \cite{TY}, is that coarser conditioning can provide
upper bounds for partial entropy.
After conditioning on fine transverse cells, we bound the
information of orbit segments by the accumulated information
of suitable coarser conditional probabilities.
Over \(n\) iterates, the comparison has a total loss of only
\(2\log n\) and holds almost surely for all sufficiently large
\(n\), by the Borel--Cantelli lemma.
Combined with transverse shrinking, this mechanism replaces
the maximal-function and differentiation estimates in the
Ledrappier--Young transverse entropy argument \cite{LY1,LY2}.
It works with H\"older holonomies in arbitrary finite transverse
dimension, remains applicable at higher regularity, and also
yields the unstable entropy formula.
The key estimate is proved in Section~\ref{sec:information}.

\subsection{Entropy formulas}\label{subsec:notation}
Let \(M\) be a compact smooth Riemannian manifold without boundary,
\(f\in\Diff^1(M)\), and \(m\) an ergodic \(f\)-invariant Borel
probability measure. Write \(h_m(f)\) for metric entropy and
\(H_m(\cdot\mid\cdot)\) for conditional entropy, with the standard
conventions of \cite{Rokhlin,KH}. All logarithms are natural.

  Let \(\lambda_1>\cdots>\lambda_u>0\) be the distinct positive
Lyapunov exponents, with corresponding Oseledets bundles
\(E^{u,j}\) \cite{Oseledets}. Put
\[
 d_j=\dim E^{u,j},\qquad
 E^{(j)}=\bigoplus_{k=1}^j E^{u,k},\qquad E^{(0)}=\{0\},
\]
and denote the sum of the remaining Oseledets bundles by
\(E^{\widehat{(j)}}\). We use \(E^u,E^c,E^s\) for the positive-,
zero-, and negative-exponent bundles, and set
\(E^{cs}=E^c\oplus E^s\), the sum of all nonpositive-exponent
bundles. Domination of \(E^u\oplus E^{cs}\) does not require
a dominated splitting inside \(E^{cs}\).

  A \(Df\)-invariant splitting \(T_\Lambda M=F\oplus G\) into
bundles of constant dimensions over an \(f\)-invariant set
\(\Lambda\) is dominated if there is an integer \(N\geqslant1\)
such that
\[
 |Df_x^N w|\leqslant\tfrac12|Df_x^N v|
\]
for every \(x\in\Lambda\) and all unit vectors
\(v\in F(x)\), \(w\in G(x)\).
More generally, \(T_\Lambda M=F_1\oplus\cdots\oplus F_r\)
is dominated if each splitting
\[
 (F_1\oplus\cdots\oplus F_\ell)
 \oplus(F_{\ell+1}\oplus\cdots\oplus F_r),
 \qquad 1\leqslant\ell<r,
\]
is dominated. Thus faster bundles are listed first.

  A dominated splitting is continuous and extends uniquely
to the closure of its base set \cite[Section~1.2]{GTY}.
For Oseledets bundles, domination on \(\supp m\) means that
the splitting on an invariant full-\(m\)-measure subset of
\(\supp m\) extends to a dominated splitting on \(\supp m\).
Zero-dimensional bundles are omitted.

  Whenever \(E^{(j)}\oplus E^{\widehat{(j)}}\) is dominated,
\cite[Proposition~8.9]{ABC}, applied to \(f^{-1}\), gives nested
\(C^1\) unstable manifolds \(W^j(x)\) tangent to \(E^{(j)}(x)\).
Write \(B^j(x,r)\) for their intrinsic balls and set
\(W^0(x)=\{x\}\). For an increasing generating measurable partition
\(\xi_j\) subordinate to \(W^j\), let \(m_x^j=m_x^{\xi_j}\) be
its conditional probabilities. The partial entropy and the lower
and upper conditional dimensions are
\begin{equation}\label{eq:partial-entropy}
 h_j=H_m(\xi_j\mid f\xi_j)=H_m(f^{-1}\xi_j\mid\xi_j)
\end{equation}
and
\begin{equation}\label{eq:partial-dimensions}
 \ldel_j=\liminf_{r\downarrow0}
       \frac{\log m_x^j(B^j(x,r))}{\log r},\qquad
 \udel_j=\limsup_{r\downarrow0}
       \frac{\log m_x^j(B^j(x,r))}{\log r}.
\end{equation}
These quantities are constant almost everywhere and independent
of the choice of subordinate partition
\cite[Sections~1.2 and~3]{Tong}. They satisfy
\[
 0\leqslant h_j\leqslant h_m(f)<\infty,\qquad
 0\leqslant\ldel_j\leqslant\udel_j\leqslant\dim E^{(j)}.
\]
Here finiteness of entropy also follows from Ruelle's inequality
\cite{Ruelle}. Set \(h_0=\ldel_0=\udel_0=0\). We call
\(h_j-h_{j-1}\) the \(j\)-th transverse entropy and write
\(\delta_j\) for the common conditional dimension when
\(\ldel_j=\udel_j\). Our first theorem identifies both dimension
increments without assuming exact dimensionality.

\begin{mainthm}\label{thm:main}
Suppose that \(2\leqslant i\leqslant u\) and
\begin{equation}\label{eq:main-splitting}
 E^{(i-1)}\oplus E^{u,i}\oplus E^{\widehat{(i)}}
\end{equation}
is dominated on \(\supp m\). Then
\begin{equation}\label{eq:main-formula}
 \ldel_i-\ldel_{i-1}
 =\udel_i-\udel_{i-1}
 =\frac{h_i-h_{i-1}}{\lambda_i}
 =:\gamma_i,\qquad 0\leqslant\gamma_i\leqslant d_i.
\end{equation}
\end{mainthm}

  In particular, exact dimensionality along \(W^{i-1}\) implies exact
dimensionality along \(W^i\). The faster bundle \(E^{(i-1)}\) need not
split into dominated single-exponent bundles.

  When \(E^u\oplus E^{cs}\) is dominated, let \(h_m^u(f)\) be
the entropy along the full unstable manifolds, as in
\cite[Section~3.1]{GTY}. Thus \(h_m^u(f)=h_u\) when \(u>0\),
and we set \(h_m^u(f)=0\) otherwise. The following theorem
identifies this entropy with \(h_m(f)\).

\begin{mainthm}\label{thm:unstable}
If \(E^u\oplus E^{cs}\) is dominated on \(\supp m\), then
\begin{equation}\label{eq:unstable-formula}
 h_m(f)=h_m^u(f).
\end{equation}
\end{mainthm}

  Applying the transverse formula to the successive positive
Oseledets bundles and using this equality gives the full entropy
formula.

\begin{mainthm}\label[mainthm]{cor:full}
Suppose that
\begin{equation}\label{eq:full-splitting}
 E^{u,1}\oplus\cdots\oplus E^{u,u}\oplus E^{cs}
\end{equation}
is dominated on \(\supp m\). For every \(1\leqslant j\leqslant u\), the conditional
measures along \(W^j\) are exact dimensional. Writing their dimension
as \(\delta_j\), and setting \(\delta_0=0\) and
\(\gamma_j=\delta_j-\delta_{j-1}\), one has
\begin{equation}\label{eq:full-formula}
 0\leqslant\gamma_j\leqslant d_j,\qquad
 h_j=\sum_{k=1}^j\lambda_k\gamma_k,\qquad
 h_m(f)=\sum_{k=1}^u\lambda_k\gamma_k.
\end{equation}
\end{mainthm}

\subsection{Measure dimension}\label{subsec:measure-dimension}

  The results below combine \cref{thm:main,thm:unstable,cor:full}
with the dimension arguments of Part~2 \cite{Tong}. We indicate
the relevant arguments and omit the parallel proofs.

  Ledrappier and Xie \cite{LX} related vanishing transverse entropy
to concentration on a single strong unstable leaf. Their
entropy-equality argument applies in the dominated \(C^1\) setting
as explained in \cite[Section~1.3]{Tong}. Together with
\cref{thm:main} and the first-level formula in
Section~\ref{subsec:full-proof}, this gives the following
characterization by conditional dimensions.
We retain the convention \(W^0(x)=\{x\}\).

\begin{mainthm}[Vanishing transverse dimension]\label{thm:zero-transverse}
Let \(0\leqslant i<j\leqslant u\), and suppose that
\begin{equation}\label{eq:interval-splitting}
 E^{(i)}\oplus E^{u,i+1}\oplus\cdots\oplus E^{u,j}
 \oplus E^{\widehat{(j)}}
\end{equation}
is dominated on \(\supp m\). The following are equivalent:
\begin{enumerate}[label=\textup{(\arabic*)},leftmargin=2em]
\item \(\udel_i=\udel_j\).
\item \(\ldel_i=\ldel_j\).
\item For every measurable partition \(\mathcal A\) subordinate to
\(W^j\), one has
\[
 m_x^{\mathcal A}(W^i(x))=1
 \quad\text{for \(m\)-almost every \(x\)}.
\]
\end{enumerate}
\end{mainthm}

  Let \(s\) be the number of distinct negative Lyapunov exponents, and
let \(E^{s,1},\ldots,E^{s,s}\) be their Oseledets bundles, ordered
by decreasing exponent. The full Oseledets splitting is
\begin{equation}\label{eq:full-oseledets}
 E^{u,1}\oplus\cdots\oplus E^{u,u}\oplus E^c
 \oplus E^{s,1}\oplus\cdots\oplus E^{s,s}.
\end{equation}
When \eqref{eq:full-oseledets} is dominated, \cref{cor:full}
applied to \(f\) and \(f^{-1}\) gives exact unstable and stable
conditional dimensions, denoted by \(d^u,d^s\). Thus, for
subordinate conditional probabilities \(m_x^u,m_x^s\),
\begin{equation}\label{eq:stable-unstable-dimensions}
 \lim_{\rho\downarrow0}
 \frac{\log m_x^*(B^*(x,\rho))}{\log\rho}=d^*,
 \qquad *\in\{u,s\},
 \quad\text{for \(m\)-almost every \(x\)}.
\end{equation}
Here \(B^*\) is an intrinsic leaf ball.

  Write \(B(x,\rho)\) for a Riemannian ball in \(M\). The lower
and upper pointwise dimensions of \(m\) are
\begin{equation}\label{eq:pointwise-dimensions}
 \ld(x)=\liminf_{\rho\downarrow0}
 \frac{\log m(B(x,\rho))}{\log\rho},
 \qquad
 \ud(x)=\limsup_{\rho\downarrow0}
 \frac{\log m(B(x,\rho))}{\log\rho}.
\end{equation}
Both functions are constant almost everywhere by invariance and
ergodicity. For the upper pointwise dimension, Tong
\cite[Theorem~1.6]{Tong} established the upper bound of
Ledrappier--Young \cite[Section~12]{LY2} in the setting of Part~2.
By the same proof, we can obtain the following theorem with our entropy formulas.

\begin{mainthm}[Upper pointwise dimension]\label{thm:dimension-upper}
If the full Oseledets splitting \eqref{eq:full-oseledets} is dominated
on \(\supp m\), then
\begin{equation}\label{eq:dimension-upper}
 \ud(x)\leqslant d^u+d^s+\dim E^c
 \quad\text{for \(m\)-almost every \(x\)}.
\end{equation}
\end{mainthm}

  The measure \(m\) is exact dimensional if
\(\ld(x)=\ud(x)\) almost everywhere. We denote the common value by
\(\dim m\). It is hyperbolic if it has no zero Lyapunov exponents.
The existence of pointwise dimension for hyperbolic measures was
conjectured by Eckmann and Ruelle \cite{ER}. Barreira, Pesin, and
Schmeling \cite{BPS} proved it for \(C^{1+\beta}\) diffeomorphisms
through an asymptotically local product structure. For dimension results
for random diffeomorphisms and endomorphisms, see Liu and Xie
\cite{LiuXie} and Shu \cite{Shu10}.
At \(C^1\) regularity, Wang and Cao \cite{WC} obtained
Hausdorff-dimension estimates under domination of the Oseledets
splitting. With the stable and unstable conditional dimensions
supplied by \cref{cor:full} for \(f\) and \(f^{-1}\), the following
is a direct application of \cite[Theorem~7.1]{Tong}.

\begin{mainthm}[Hyperbolic exact dimensionality]\label{thm:exact-dimensional}
If \(m\) is hyperbolic and its full Oseledets splitting is dominated
on \(\supp m\), then \(m\) is exact dimensional and
\begin{equation}\label{eq:exact-dimension}
 \ld(x)=\ud(x)=\dim m=d^u+d^s
 \quad\text{for \(m\)-almost every \(x\)}.
\end{equation}
\end{mainthm}

\subsection{Invariance principles}\label{subsec:invariance}

  Invariance principles relate Lyapunov exponents to holonomy
invariance of conditional measures, following Ledrappier's work
on linear cocycles \cite{LedrappierCocycles} and Avila and Viana's
extension to smooth cocycles \cite{AV}.

  Throughout this subsection, \(f\in\Diff^1(M)\) is dynamically
coherent and partially hyperbolic. Its dominated splitting is
\(TM=F^u\oplus F^c\oplus F^s\), with all three bundles nonzero,
\(F^u\) uniformly expanding, and \(F^s\) uniformly contracting.
Denote the strong unstable, center, and strong stable foliations
by \(\mathcal F^u,\mathcal F^c,\mathcal F^s\), and the
center--unstable and center--stable foliations by
\(\mathcal F^{cu},\mathcal F^{cs}\). Write
\(h_m(f,\mathcal F^u)\) for entropy along \(\mathcal F^u\).

\begin{definition}[Compact-center setting]\label{def:compact-center}
The hypotheses of Tahzibi and Yang
\cite[Section~2, \textup{(H1)--(H3)}]{TY} require that every center
leaf is compact, that the stable and unstable holonomies
\(H^\sigma_{x,y}\), \(y\in\mathcal F^\sigma(x)\),
\(\sigma\in\{s,u\}\), extend to homeomorphisms
\(\mathcal F^c(x)\to\mathcal F^c(y)\), and that \(f\) induces
a transitive topological Anosov homeomorphism \(f_c\) on
\(M_c=M/\mathcal F^c\).
\end{definition}

  We denote the quotient projection by \(\pi_c:M\to M_c\), so
\(\pi_c\circ f=f_c\circ\pi_c\). For \(C^2\) diffeomorphisms in
this setting, Tahzibi and Yang
\cite[Theorem~A and Proposition~5.4]{TY} proved
\[
 h_m(f,\mathcal F^u)\leqslant h_{(\pi_c)_*m}(f_c),
\]
with equality precisely when the center disintegration is
invariant under unstable holonomy. The factor inequality also
gives \(h_{(\pi_c)_*m}(f_c)\leqslant h_m(f)\); see the relative
variational principle of Ledrappier and Walters \cite{LW}.
Crovisier and Poletti \cite{CP} subsequently established a
\(C^1\) entropy criterion in the following setting, which allows
noncompact center leaves.

\Needspace{9\baselineskip}
\begin{definition}[Quasi-isometric center]\label{def:quasi-isometric-center}
Following Crovisier and Poletti \cite[Definition~1.1]{CP}, we say
that \(f\) is quasi-isometric in the center if there are constants
\(K\geqslant1\) and \(b>0\) such that
\[
 K^{-1}d^c(x,y)-b
 \leqslant d^c(f^nx,f^ny)
 \leqslant Kd^c(x,y)+b
\]
for every \(n\in\Z\) and all \(x,y\) on the same center leaf,
where \(d^c\) is the intrinsic distance on that leaf.
\end{definition}

  The compact-center setting of \cref{def:compact-center} is a
special case. Its Hausdorff leaf space, the compact-foliation
results of \cite[Theorems~4.1--4.3]{Epstein}, and uniform local
plaque geometry on the compact manifold give uniformly bounded
intrinsic center-leaf diameters. Thus \(K=1\) and a sufficiently
large \(b\) suffice.
The quasi-isometric setting also includes the discretized Anosov
flows considered in \cite{CP}.

  Write \(m_x^c\) for the leafwise center conditional measures:
normalized probabilities in the compact-center setting and
Radon measures defined up to a positive scalar otherwise.
For \(x,y\) in the full-measure set of \cite[Theorem~C]{CP}
with \(\mathcal F^{cu}(x)=\mathcal F^{cu}(y)\), let
\(H^u_{x,y}:\mathcal F^c(x)\to\mathcal F^c(y)\) be holonomy
along \(\mathcal F^u\). It sends \(z\) to the unique point of
\(\mathcal F^u(z)\cap\mathcal F^c(y)\).
Following \cite[Section~4.1]{TY} and \cite[Section~1.2]{CP}, the
center disintegration is \(u\)-invariant if, on almost
every local center--unstable plaque,
\[
 (H^u_{x,y})_*m_x^c=c_{x,y}m_y^c,\qquad c_{x,y}>0,
\]
for almost every pair of center plaques with respect to
the product of the induced quotient measure.
In the compact-center setting, normalization gives
\(c_{x,y}=1\).

\begin{corollary}[\(C^1\) invariance principle]\label[corollary]{cor:invariance}
Let \(f\) be quasi-isometric in the center as in
\cref{def:quasi-isometric-center}, and let \(m\) be an ergodic
\(f\)-invariant probability measure. If all Lyapunov exponents
of \(m\) along \(F^c\) are nonpositive, then its center
disintegration is \(u\)-invariant.
\end{corollary}

  This proves the \(C^1\) extension of \cite[Corollary~D]{CP}
proposed by Crovisier and Poletti in footnote~2, extending
Part~1 \cite[Section~6.1]{GTY} to arbitrary center dimension.
Indeed,
\(E^u=F^u\) and \(E^{cs}=F^c\oplus F^s\) almost everywhere.
Partial hyperbolicity gives domination of
\(F^u\oplus(F^c\oplus F^s)\), so \cref{thm:unstable} yields
\(h_m(f)=h_m(f,\mathcal F^u)\). The \(C^1\) entropy criterion
of Crovisier and Poletti \cite[Main Theorem]{CP} gives the
conclusion.

\subsection{Related work}

  Related results for expanding foliations include the unstable metric
and topological entropy of Hu, Hua, and Wu \cite{HHW}, Yang's
continuity results and applications \cite{Yang}, and the
thermodynamical \(u\)-formalism and transverse measures of Ures,
Viana, Yang, and Yang \cite{UVYYu,UVYYTransverse}.

  Liao, Tao, Tong, and Yang \cite{LTTY} established criteria for
upper semicontinuity of partial entropy for \(C^{1+\beta}\)
diffeomorphisms in terms of continuity of Lyapunov sums; their
Section~1.2.1 and Remark~6 give a \(C^1\) version under dominated
splitting assumptions.
Ding, Gan, and Tong \cite{DGT} proved an entropy characterization
of SRB measures under domination for
\(C^{1+\mathrm{Dini}}\) diffeomorphisms. The \(C^1\) examples
recalled in \cite[Section~6]{DGT} have uniformly hyperbolic
attractors without SRB measures. The entropy--dimension identities proved
here do not assert absolute continuity of unstable conditional
measures.

  Further approaches to entropy and dimension include neutralized local
entropy, introduced by Ben Ovadia and Rodriguez-Hertz \cite{BOHR},
and the \(r\)-neutralized entropy formulas of Dong and Qiao \cite{DQ}.
Ben Ovadia studied leaf measures for hyperbolic equilibrium states
\cite{BOLeaf} and introduced tubular dimension to describe leafwise
asymptotic local product structure \cite{BOTubular}. For approximation
of measure dimension by dimensions of hyperbolic sets, see Cao, Wang,
and Zhao \cite{CWZ}.

  Beyond invertible deterministic dynamics, Qian, Qian, and Xie
\cite{QQX} proved an entropy formula for random systems. For
endomorphisms, see the entropy formulas of Qian and Zhu \cite{QZ},
Qian and Xie \cite{QX}, and Shu \cite{Shu09}, and Liu's relations
between metric entropy, folding entropy, and negative Lyapunov
exponents \cite{Liu03,Liu08}.

\subsection{Outline of the proof}
The lower bound for transverse entropy is the cylinder-counting
estimate of Part~2, recalled in Section~\ref{sec:counting}.
For the reverse bound, we use adapted increasing partitions
along \(W^i\) and \(W^{i-1}\), together with coordinates
transverse to the strong unstable leaves. These are constructed
in Sections~\ref{sec:geometry}--\ref{sec:partitions}.

We cut the atoms of the \(W^i\)-partition by a transverse
grid whose scale decreases exponentially with the orbit length.
Projection along strong unstable leaves controls transverse
growth by the intermediate Lyapunov exponent. If the initial
grid scale decreases sufficiently fast, the transverse diameters
of the transported cells tend to zero uniformly over the orbit
segment, for almost every initial point. This is the geometric
input established in Section~\ref{subsec:transverse-cells}.

The new step is to estimate the conditional mass of a
transverse cell. Inside each initial cell, products of suitable
coarse conditional probabilities give weights on successive
orbit cylinders with total mass at most one. An elementary
probability comparison bounds the true cylinder information
conditioned on that cell by the accumulated comparison information
plus \(2\log n\), almost everywhere for all sufficiently large
\(n\), where \(n\) is the orbit length. Transverse shrinking
makes the contribution from visits near the coarse grid boundary
negligible, and refining the coarse grid makes the resulting
asymptotic upper bound converge to \(h_{i-1}\).
The same cylinders have information rate \(h_i\) when conditioned
only on the \(W^i\)-atom. Comparing the two conditional
probabilities through the chain rule gives a lower bound
\(h_i-h_{i-1}\) for the exponential decay rate of transverse
cell masses. This is the mass estimate in
Section~\ref{sec:information}.
Shifted grids turn it into a transverse dimension bound, and
the dimension inequalities of Part~2 and Ledrappier--Young
transfer that bound to the intrinsic conditional dimensions.
Section~\ref{sec:consequences} concludes the transverse formula,
applies the same mass estimate to the nonpositive-exponent bundle,
and obtains the full formula by induction.

\medskip
\noindent\textbf{Acknowledgements.}
The author thanks Shaobo Gan and Jiagang Yang for helpful
discussions. The author also thanks Jiagang Yang for
suggesting the project developed in this paper and
Parts~1 and~2 at the beginning of the author's doctoral studies.

\section{Geometric preliminaries}\label{sec:geometry}

  Let \(d\) and \(d^j\) denote the Riemannian distance on \(M\)
and the intrinsic distance on \(W^j\), respectively.
Fix an index \(i\) satisfying \eqref{eq:main-splitting}. We use the
fake foliations, local unstable plaques, and straightening maps from
\cite[Sections~2--3]{Tong} and
\cite[Sections~2--3 and~5.1]{GTY}. We record the properties needed below and refer to those papers
for the constructions.

  Fix \(M\) with its metric, \(f\), \(m\), the index \(i\), and
the dominated splitting \eqref{eq:main-splitting}, together with
a domination time. There is a threshold \(\eps_*>0\), depending
only on these choices, for the constructions below. Fix
\[
 0<\eps<\min\left\{\eps_*,\frac{\lambda_{i-1}-\lambda_i}{7},
                         \frac{\lambda_i}{4}\right\},\qquad
 \alpha=1-\frac{7\eps}{\lambda_{i-1}-\lambda_i}\in(0,1).
\]

  The underlying geometric framework is the invariant-manifold and
plaque-family theory of Hirsch, Pugh, and Shub \cite{HPS}; related
fake-foliation constructions appear in \cite{BW,LVY}. The leaves \(\tW_x^*(v)\) below are subsets of \(T_xM\), whereas
\(W^j(y)\) is a subset of \(M\). We use \(\exp_x\) to compare them.

\subsection{Fake foliations and their growth estimates}\label{subsec:fake-foliations}

  We use the cutoff and graph-transform construction of
\cite[Section~2.1]{GTY}, as in \cite[Sections~2.1 and~7.2.1]{Tong}.
All objects below are defined on a common invariant Borel set
\(\Gamma'\subset\Gamma_{\mathrm{reg}}\cap\supp m\) of full measure,
where \(\Gamma_{\mathrm{reg}}\) is the Oseledets regular set.
For \(v=v_{i-1}+v_i+v_{\widehat i}\) in the three summands of
\eqref{eq:main-splitting}, set
\begin{equation}\label{eq:box-norm}
 |v|'=\max\{|v_{i-1}|,|v_i|,|v_{\widehat i}|\},\qquad
 R_x(r)=\{v\in T_xM:|v|'\leqslant r\}.
\end{equation}
Here \(|\cdot|\) is the Riemannian norm. The angles of the
splitting give a uniform constant \(C_1\geqslant1\) such that
\(C_1^{-1}|v|\leqslant|v|'\leqslant C_1|v|\).
On sums of the three subspaces, we use the induced maximum norm.

  For a sufficiently small uniform radius \(r_0>0\), the local map
\[
 \widetilde f_x=\exp_{fx}^{-1}\circ f\circ\exp_x
\]
is defined on \(\{v\in T_xM:|v|\leqslant2C_1r_0\}\) for every
\(x\in\supp m\), and \(D_0\widetilde f_x=Df_x\).
Choose a smooth function \(\chi:[0,\infty)\to[0,1]\) with
\(\chi=1\) on \([0,1]\) and \(\chi=0\) on \([2,\infty)\).
Cut off the nonlinear part of \(\widetilde f_x\) to define
\(\tg_x:T_xM\to T_{fx}M\) by
\begin{equation}\label{eq:cutoff-extension}
 \tg_x(v)=
 \begin{cases}
 Df_xv+\chi\!\left(\dfrac{|v|}{C_1r_0}\right)
          \bigl(\widetilde f_x(v)-Df_xv\bigr),
       & |v|\leqslant2C_1r_0,\\[2pt]
 Df_xv, & |v|>2C_1r_0.
 \end{cases}
\end{equation}
The cutoff estimates in \cite[Section~2.1]{GTY} give
\[
 \sup_{x\in\supp m}\sup_{v\in T_xM}
 \|D_v\tg_x-Df_x\|\longrightarrow0
 \qquad(r_0\downarrow0).
\]
Thus, after decreasing \(r_0\), each \(\tg_x\) is a global
\(C^1\) diffeomorphism, with \(\tg_x(0)=0\) and
\(D_0\tg_x=Df_x\). The factor \(C_1\) in the cutoff ensures
\(\tg_x=\widetilde f_x\) on the box \(R_x(r_0)\).
For \(n\geqslant1\), write
\[
 \tg_x^n=\tg_{f^{n-1}x}\circ\cdots\circ\tg_x,\qquad
 \tg_x^0=\mathrm{id},\qquad
 \tg_x^{-n}=(\tg_{f^{-n}x}^n)^{-1}.
\]
These iterates are defined on the whole tangent space. Along
orbit segments staying in the agreement boxes, they coincide
with the iterates of \(f\) in exponential coordinates.

  A complete graph over \(F\subset T_xM\), relative to a
complement \(G\), is a set
\(\{w+\psi(w):w\in F\}\), where \(\psi:F\to G\) is \(C^1\).
Its slope is \(\sup_w\|D\psi(w)\|\), measured in the splitting
norms above. A fake foliation modeled on \(F\) is a partition
of \(T_xM\) into such graphs. We use the four bundles in the two
complementary pairs
\[
 E^{(i-1)}\oplus E^{\widehat{(i-1)}},\qquad
 E^{(i)}\oplus E^{\widehat{(i)}},\qquad
 \text{ with }\  E^{\widehat{(i-1)}}=E^{u,i}\oplus E^{\widehat{(i)}}.
\]
The complement of each bundle is the other member of its pair.
Denote the leaf through \(v\) of the foliation modeled on
\(E_x^*\) by \(\tW_x^*(v)\). The graph-transform construction
\cite[Lemma~2.2]{Tong} gives the following properties.

\begin{lemma}[Fake foliations]\label[lemma]{lem:fake-foliations}
For every \(0<\theta\leqslant1/4\), one can choose \(r_0>0\)
sufficiently small so that there are fake foliations modeled
on the four bundles above, with leaves that are complete
\(C^1\) graphs of slope at most \(\theta/10\). For every \(x\in\Gamma'\) and \(v\in T_xM\),
\[
 \tg_x\bigl(\tW_x^*(v)\bigr)
 =\tW_{fx}^*(\tg_xv),
 \qquad *\in\{(i-1),(i),\widehat{(i)},\widehat{(i-1)}\}.
\]
The \((i-1)\)-foliation subfoliates the \((i)\)-foliation, so
\(\tW_x^{(i-1)}(v)\subset\tW_x^{(i)}(v)\).
In local bundle trivializations, the leaf graphs depend
continuously on \((x,v)\) in the topology of uniform convergence
on compact subsets of their graph domains.
\end{lemma}

  Define the fake \(i\)-leaves by
\begin{equation}\label{eq:fake-center}
 \tW_x^i(v)=\tW_x^{(i)}(v)\cap\tW_x^{\widehat{(i-1)}}(v).
\end{equation}
As in \cite[Section~2.1]{Tong}, they are complete graphs over
\(E_x^{u,i}\), relative to \(E_x^{(i-1)}\oplus E_x^{\widehat{(i)}}\),
and are invariant under \(\tg_x\). Their slopes are at most
\(\theta\), also in the \(E^{(i-1)}\times E^{u,i}\) graph coordinates
of \(\tW_x^{(i)}\). We call \(\tW_x^i(0)\) the standard fake
\(i\)-leaf; it is the \(d_i\)-dimensional transversal used below.
In contrast, the genuine unstable manifold \(W^i(x)\) is tangent
to the bundle \(E^{(i)}\) and corresponds locally to
\(\tW_x^{(i)}(0)\), not to \(\tW_x^i(0)\).

  The length estimates follow from the block estimates of
\cite[Lemma~2.1]{Tong} by the local argument of
\cite[Remark~3]{GTY}. The radius and slope bounds above can be
chosen so that they hold simultaneously with the fake-foliation
properties. Since all exponents on \(E^{u,i}\) equal \(\lambda_i\),
the maximal exponents of the forward and backward cocycles on
this bundle are \(\lambda_i\) and \(-\lambda_i\), respectively.

\begin{lemma}[Length bounds]\label[lemma]{lem:fake-lengths}
There are a Borel function \(A:\Gamma'\to[1,\infty)\) and a
constant \(C_2\geqslant1\) such that
\begin{equation}\label{eq:tempered}
 A(f^{\pm1}x)\leqslant e^\eps A(x).
\end{equation}
For every \(x\in\Gamma'\), \(n\geqslant0\), and rectifiable
curve \(\gamma\) in a fake \(i\)-leaf in \(T_xM\),
\begin{align}
 \operatorname{length}(\tg_x^n\gamma)
 &\leqslant C_2 A(x)e^{(\lambda_i+3\eps)n}
                 \operatorname{length}(\gamma),\label{eq:F-forward}\\
 \operatorname{length}(\tg_x^{-n}\gamma)
 &\leqslant C_2 A(x)e^{(-\lambda_i+3\eps)n}
                 \operatorname{length}(\gamma).\label{eq:F-backward}
\end{align}
For \(j=i-1,i\), a rectifiable curve \(\gamma\) in a fake
\((j)\)-leaf in \(T_xM\) satisfies
\begin{equation}\label{eq:strong-backward}
 \operatorname{length}(\tg_x^{-n}\gamma)
 \leqslant C_2 A(x)e^{(-\lambda_j+3\eps)n}
                 \operatorname{length}(\gamma).
\end{equation}
Lengths use the Riemannian norms on the tangent spaces, and
\(C_2\) is independent of \(x,n,\gamma\).
\end{lemma}

  Fix a common choice of \(r_0,\theta,A,C_2\) for the chosen
\(\eps\), before choosing \(\delta\). Numbered constants retain
their values; they may depend on these choices, but not on orbit
length. Dependence on a base point is indicated by a subscript.

\subsection{Holonomy and local unstable manifolds}

  An \(i\)-transversal is an open subset of a complete \(C^1\) graph
over \(E_x^{u,i}\), with slope at most the fixed \(\theta\), contained in
\(\tW_x^{(i)}(0)\). It has dimension \(d_i\). For two such transversals
\(U,V\), holonomy along the fake strong leaves is the map
\[
 \widetilde h_x^{(i-1)}:U\longrightarrow V,\qquad
 \{\widetilde h_x^{(i-1)}(p)\}=\tW_x^{(i-1)}(p)\cap V,
\]
on a domain where the intersection exists uniquely and the map is
injective. Its inverse is the holonomy in the opposite direction.
A map is \(\alpha\)-bi-H\"older if it is a homeomorphism onto its image
and both it and its inverse are \(\alpha\)-H\"older. The proof of
\cite[Lemma~2.4]{GTY}, with spectral gap
\(\lambda_{i-1}-\lambda_i\), gives the following
higher-dimensional form of \cite[Lemma~2.3]{Tong}.

\begin{lemma}[Uniform holonomy estimate]\label[lemma]{lem:holonomy-input}
For every \(K\geqslant1\) and \(R>0\) there is a constant
\(C_{3}(K,R)\geqslant1\) with the following property. If
\(x\in\Gamma'\), \(A(x)\leqslant K\), \(U,V\) are \(i\)-transversals,
\(\diam U\leqslant R\), and
\(\sup_{p\in U}|p-\widetilde h_x^{(i-1)}(p)|\leqslant R\), then
\begin{equation}\label{eq:holonomy}
 C_{3}(K,R)^{-1}|p-q|^{1/\alpha}
 \leqslant|\widetilde h_x^{(i-1)}(p)-\widetilde h_x^{(i-1)}(q)|
 \leqslant C_{3}(K,R)|p-q|^\alpha
\end{equation}
for \(p,q\in U\). The constant is uniform over these base points,
transversals, and holonomy domains.
\end{lemma}
\begin{proof}
Apply the proof of \cite[Lemma~2.4]{GTY} with the spectral gap
\(\lambda_{i-1}-\lambda_i\), as in \cite[Lemma~2.3]{Tong}.
The argument uses distances and length estimates on graphs
and applies to \(d_i\)-dimensional transversals without change.
It gives the bounds with exponent
\(\alpha=1-7\eps/(\lambda_{i-1}-\lambda_i)\), uniformly under the
stated bounds on \(A(x)\) and the holonomy displacement.
\end{proof}

  For \(x\in\Gamma'\), projection along the leaves
\(\tW_x^{(i-1)}\) onto the standard fake \(i\)-leaf is defined by
\begin{equation}\label{eq:strong-projection}
 \Pi_x^i:\tW_x^{(i)}(0)\longrightarrow\tW_x^i(0),\qquad
 \{\Pi_x^i(v)\}=\tW_x^{(i-1)}(v)\cap\tW_x^i(0).
\end{equation}
Its restriction to an \(i\)-transversal is holonomy along
\(\tW_x^{(i-1)}\).

  To verify that this projection is well defined, use graph coordinates
on \(\tW_x^{(i)}(0)\) and write
\(\iota_x(u,t)=u+t+b_x(u,t)\), where
\(u\in E_x^{(i-1)}\), \(t\in E_x^{u,i}\),
\(b_x(u,t)\in E_x^{\widehat{(i)}}\), \(b_x(0,0)=0\), and
\(\operatorname{Lip}b_x\leqslant\theta\).
The standard transversal and the strong leaf through
\(v=\iota_x(u_0,t_0)\) are respectively
\[
 \{\iota_x(\psi(t),t):t\in E_x^{u,i}\},\qquad
 \{\iota_x(u,\varphi(u)):u\in E_x^{(i-1)}\},
\]
with \(\psi(0)=0\), \(\varphi(u_0)=t_0\), and both Lipschitz
constants at most \(\theta<1\). The intersection corresponds to
the unique fixed point \(t_*\) of the \(\theta^2\)-contraction
\(t\mapsto\varphi(\psi(t))\). Since
\[
 |t_*|\leqslant(1+\theta)|v|'+\theta^2|t_*|,
\]
we have \(|t_*|\leqslant|v|'/(1-\theta)\). The graph slope bounds
then give
\begin{equation}\label{eq:projection-bound}
 |\Pi_x^i(v)|'\leqslant C_4|v|',\qquad C_4=(1-\theta)^{-1}.
\end{equation}
Uniqueness shows that \(\Pi_x^i\) fixes \(\tW_x^i(0)\) pointwise and
\begin{equation}\label{eq:projection-fibers}
 \Pi_x^i(v)=\Pi_x^i(w)
 \quad\Longleftrightarrow\quad
 \tW_x^{(i-1)}(v)=\tW_x^{(i-1)}(w).
\end{equation}
Invariance of both fake foliations gives the equivariance
\begin{equation}\label{eq:projection-equivariance}
 \Pi_{fx}^i\circ\tg_x=\tg_x\circ\Pi_x^i
 \quad\text{on }\tW_x^{(i)}(0).
\end{equation}

  For \(\delta>0\), the \(\delta\)-fake foliation chart at \(x\) is the pair
\((R_x(\rho_\delta(x)),\exp_x)\), where
\begin{equation}\label{eq:rho}
 \rho_\delta(x)=\delta A(x)^{-1}.
\end{equation}
For \(j=i-1,i\) and a regular point \(y\) in the chart at \(x\),
\(W^j_{x,\delta}(y)\) denotes the inverse exponential image of the
connected component through \(y\), in the intrinsic leaf topology, of
\(W^j(y)\cap\exp_xR_x(\rho_\delta(x))\). The center \(x\) determines
the chart radius even when the plaque passes through \(y\).
Tong \cite[Lemmas~2.4--2.5]{Tong} identifies these local plaques using
the two coarsenings of \eqref{eq:main-splitting}.

\begin{lemma}[Identification of local plaques]\label[lemma]{lem:plaque-input}
With the fake foliations and length bounds in
\cref{lem:fake-foliations,lem:fake-lengths} fixed, there is
\(\delta_*>0\) such that for
\(0<\delta\leqslant\delta_*\) and \(x\in\Gamma'\), the boxes
\(R_x(2\rho_\delta(x))\) lie in the
domains of agreement of the extensions with \(f\), and
\begin{equation}\label{eq:positive-plaque}
 W^i_{x,\delta}(x)=\tW_x^{(i)}(0)\cap R_x(\rho_\delta(x)).
\end{equation}
For every regular \(y\in\exp_xW^i_{x,\delta}(x)\),
\begin{equation}\label{eq:strong-plaque}
 W^{i-1}_{x,2\delta}(y)
 =\tW_x^{(i-1)}(\exp_x^{-1}y)\cap R_x(2\rho_\delta(x)).
\end{equation}
\end{lemma}

  Fix such a \(\delta\). Write a leaf of \(\tW_x^{(i-1)}\) as the graph
\(\{u+\psi_s(u):u\in E_x^{(i-1)}\}\), where
\(\psi_s:E_x^{(i-1)}\to E_x^{\widehat{(i-1)}}\) has a Lipschitz
constant of at most \(1/4\). For points
\(p=u_p+\psi_s(u_p)\), \(q=u_q+\psi_s(u_q)\) in \(R_x(r)\), put
\(u_t=(1-t)u_p+tu_q\), \(0\leqslant t\leqslant1\). Then
\begin{equation}\label{eq:graph-segment}
 |u_t|\leqslant r,\qquad
 |\psi_s(u_t)|'\leqslant|\psi_s(u_p)|'+\tfrac14|u_t-u_p|
 \leqslant\tfrac32r.
\end{equation}
Thus \(t\mapsto u_t+\psi_s(u_t)\) joins \(p\) to \(q\) on that leaf
inside \(R_x(3r/2)\).

\subsection{Straightening maps}\label{subsec:straightening}

  Fix \(A_0<\infty\) with \(m\{A\leqslant A_0\}>0\), and choose
\(x_*\in\Gamma'\cap\{A\leqslant A_0\}\) to be an \(m\)-density
point of this set. In the exponential chart at \(x_*\), use
coordinates \((u,t,w)\) along the three summands of
\eqref{eq:main-splitting}, identified with
\[
 \R^{(i-1)}\times\R^{d_i}\times
 \R^{\dim E^{\widehat{(i)}}},\qquad
 \R^{(i-1)}=\R^{\dim E^{(i-1)}}.
\]
We also write \(u,t,w\) for the coordinate maps on the chart.

  The box construction of \cite[Section~2.3]{Tong}, following
\cite[Sections~8.3--8.4]{LY2}, gives concentric chart boxes
\(U\Subset U^+\) and a Borel set
\[
 G\subset\Gamma'\cap\{A\leqslant A_0\}\cap U,\qquad m(G)>0.
\]
The local \(W^i\)-disks \(\exp_zW^i_{z,\delta}(z)\), \(z\in G\),
are graphs over a common \((u,t)\)-box in \(U^+\). Their
intersections with \(U\) are connected and either coincide or
are disjoint. Set
\[
 S=\bigcup_{z\in G}
   \bigl(\exp_zW^i_{z,\delta}(z)\cap U\bigr)
   =\bigsqcup_L L,
\]
where \(L\) ranges over the distinct intersections. These are
the \(W^i\)-plaques of \(S\). For a disk representing \(L\), put
\[
 L^+=\exp_zW^i_{z,\delta}(z)\cap U^+,\qquad L=L^+\cap U.
\]
Here \(L^+\) is taken from the same local disk. Write \(L_x\) for
the plaque containing \(x\in S\). Since \(A\leqslant A_0\) on
\(G\), the boxes can be chosen so that
\(L_x^+\subset\exp_xW^i_{x,\delta}(x)\) for every \(x\in G\).

  The local \(W^{i-1}\)-plaques through points of \(L\), taken
in \(L^+\), are graphs over a common \(u\)-box containing \(0\).
These plaque families are obtained from the fake foliations
and agree with the genuine \(W^{i-1}\)-plaques at regular points
by \cref{lem:plaque-input}. Their graph slopes, and those of
\(L^+\), are uniformly small. Intersecting these
\(W^{i-1}\)-plaques with \(L\) gives its finer plaque partition.
The set \(S\) is Borel, and both plaque partitions are measurable
with standard Borel spaces of labels.

  Define \(T_L=L^+\cap\{u=0\}\). For \(y\in L\), let \(H_L(y)\)
be the unique point of its local \(W^{i-1}\)-plaque in \(T_L\),
obtained by setting \(u=0\) in the graph. Local leaf compatibility
makes this point independent of the disk representing \(L\).
The point \(H_L(y)\) may lie outside \(L\), but it remains in
\(L^+\), where the fixed chart is defined. The graph bounds make
\(t|_{T_L}\) uniformly bi-Lipschitz. Set
\begin{equation}\label{eq:straightening}
 \begin{split}
 \pi=(\pi^{(i-1)},\pi^i)&:
 S\longrightarrow\R^{(i-1)}\times\R^{d_i},\\
 \pi(y)&=\bigl(u(y),t(H_L(y))\bigr),\qquad y\in L.
 \end{split}
\end{equation}
On each \(L\), the fibers of \(\pi^i\) are its
\(W^{i-1}\)-plaques, and \(\pi^{(i-1)}\) gives the coordinate
along each fiber. The proof of \cite[Lemma~2.6]{Tong}, with
\(t\in\R^{d_i}\), and \cref{lem:holonomy-input} give the
following regularity.

\begin{proposition}[Straightening maps]\label[proposition]{prop:box-input}
The box \(S\) can be chosen arbitrarily small, with
\(m(S\cap\Gamma'\cap\{A\leqslant A_0\})>0\), so that \(\pi\)
is bounded and Borel. On each \(W^i\)-plaque \(L\),
\(\pi|_L\) is an injective \(\alpha\)-bi-H\"older map, and its
restriction to each \(W^{i-1}\)-plaque is bi-Lipschitz. These
estimates hold with a common constant \(C_5\geqslant1\),
independent of \(L\), in local Riemannian and Euclidean coordinates.
\end{proposition}

  The map \(\pi\) need not distinguish different \(W^i\)-plaques.
Its transverse component \(\pi^i\) uses the fixed chart, whereas
\(\Pi_x^i\) takes values in \(\tW_x^i(0)\subset T_xM\).
Section~\ref{subsec:return-domains} compares these two descriptions
of the local \(W^{i-1}\)-classes on return domains.

\section{Partitions and transverse coordinates}\label{sec:partitions}

  We use Rokhlin measurable partitions and their conditional
probabilities with the standard conventions of \cite{Rokhlin,KH}.
All partitions and identities are considered modulo \(m\)-null
sets. For a partition \(\zeta\), write \(\zeta(x)\) for its atom,
\(\sigma(\zeta)\) for its completed sigma-algebra, and
\(m_x^\zeta\) for its conditional probability, so that
\(m_x^\zeta(\zeta(x))=1\) almost everywhere. The notation
\(\eta\prec\xi\) means that \(\xi\) refines \(\eta\), allowing
equality, and \(\eta\vee\xi\) is their common refinement.
For \(k\in\Z\),
\[
 (f^k\zeta)(x)=f^k\bigl(\zeta(f^{-k}x)\bigr).
\]
Invariance gives
\begin{equation}\label{eq:conditional-equivariance}
 m_{f^kx}^{f^k\zeta}=(f^k)_*m_x^\zeta
 \qquad(k\in\Z),\quad\text{for almost every }x.
\end{equation}
An increasing partition satisfies \(f\zeta\prec\zeta\).
Subordination to \(W^j\) means that, for almost every \(x\),
\(\zeta(x)\subset W^j(x)\) and \(\zeta(x)\) contains an intrinsic
neighborhood of \(x\). The partition is generating if
\(\bigvee_{n\geqslant0}f^{-n}\zeta\) is the partition into points. We write \(\B(Y)\) for the Borel
sigma-algebra of \(Y\). Finite conditional entropy
\(H_m(\mathcal U\mid\mathcal V)<\infty\) ensures that, on almost
every \(\mathcal V\)-atom, the intersections of positive conditional mass with the atoms
of \(\mathcal U\) form a countable partition modulo the conditional
probability.

  We construct nested subordinate partitions \(\xi_i\prec\xi_{i-1}\),
their adapted refinements \(\eta_i\prec\eta_{i-1}\), and a coordinate
\(\tp^i\) satisfying
\(\eta_{i-1}(x)=\{y\in\eta_i(x):\tp^i(y)=\tp^i(x)\}\).

  In each argument we use a common invariant Borel set
\(X\subset\Gamma'\) of full \(m\)-measure on which the stated
partition identities hold. The conditional probabilities for the countably many partitions
used in that argument assign full mass to \(X\) on almost every
atom. Assertions about every regular point of an atom, and all
atom diameters below, refer to its intersection with \(X\).

\subsection{Subordinate partitions and return domains}\label{subsec:return-domains}

  Inside \(S\), let \(\widehat\xi_i\) have the plaques \(L\) as
its atoms, and let \(\widehat\xi_{i-1}\) refine each \(L\) into
the fibers of \(\pi^i|_L\), namely its local \(W^{i-1}\)-plaques.
Complete both partitions by the atom \(M\setminus S\), and define
\begin{equation}\label{eq:subordinate}
 \xi_j=\bigvee_{n\geqslant0}f^n\widehat\xi_j,\qquad j=i-1,i.
\end{equation}
The recurrence construction in \cite[Section~3]{Tong}, following
\cite[Section~3.1]{GTY}, gives the following choice of the box.

\begin{lemma}[Subordinate partitions]\label[lemma]{lem:subordinate-input}
The box in \cref{prop:box-input} can be chosen arbitrarily small so
that \(\xi_{i-1}\) and \(\xi_i\) are increasing, generating, and
subordinate to \(W^{i-1}\) and \(W^i\), respectively. Moreover,
\(\xi_i\prec\xi_{i-1}\).
\end{lemma}

  Fix a joint choice in \cref{prop:box-input,lem:subordinate-input} with
\(\diam S<\delta/A_0\). Fix the chart boxes, the plaque families,
the map \(\pi\), and the subordinate partitions for this choice.
By \eqref{eq:partial-entropy},
\[
 H_m(f^{-1}\xi_j\mid\xi_j)=h_j,\qquad j=i-1,i.
\]

  We now express \(\pi^i(y)\) in terms of the strong representative
\(\Pi_x^i(\exp_x^{-1}y)\) on uniform return domains.
Use the inner-box choice in \cite[Section~5.1]{GTY} and
\cite[Sections~2.3 and~5.2.1]{Tong} to choose a positive-measure
Borel set
\begin{equation}\label{eq:return-set}
 E\subset G\subset S\cap\Gamma'\cap\{A\leqslant A_0\}
\end{equation}
and a radius \(0<r_*\leqslant\delta/(4C_4A_0)\). The choice
of the outer box in Section~\ref{subsec:straightening} and
\cref{lem:plaque-input} give
\(\exp_x^{-1}L_x^+\subset\tW_x^{(i)}(0)\) for \(x\in E\). Put
\begin{equation}\label{eq:representative-coordinate-domains}
 \begin{split}
 U_x^i&=E\cap L_x\cap
       \exp_x\bigl(\tW_x^{(i)}(0)\cap R_x(r_*)\bigr),\\
 V_x^i&=\Pi_x^i(\exp_x^{-1}U_x^i)\subset\tW_x^i(0).
 \end{split}
\end{equation}
The graph and local plaque properties make
\(T_x=\exp_x^{-1}(T_{L_x})\) an \(i\)-transversal in
\(\tW_x^{(i)}(0)\). Each local \(W^{i-1}\)-plaque meeting
\(U_x^i\) reaches \(T_{L_x}\) within this chart. For
\(v\in V_x^i\), define
\[
 \{h_x(v)\}=\tW_x^{(i-1)}(v)\cap T_x,\qquad
 \kappa_x^i(v)=t\bigl(\exp_x h_x(v)\bigr).
\]
Here \(t\) is the same fixed coordinate used in
\eqref{eq:straightening}.

\begin{lemma}[Coordinates on strong representatives]\label[lemma]{lem:representative-coordinates}
For every \(x\in E\), the map \(\kappa_x^i\) is a bijection from
\(V_x^i\) onto \(\pi^i(U_x^i)\), with
\begin{equation}\label{eq:coordinate-factorization}
 \pi^i(y)=\kappa_x^i\bigl(\Pi_x^i(\exp_x^{-1}y)\bigr),
 \qquad y\in U_x^i.
\end{equation}
Moreover, \(\kappa_x^i(0)=\pi^i(x)\), and there is a constant
\(C_6\geqslant1\), independent of \(x\), such that
\begin{equation}\label{eq:representative-coordinate-holder}
 C_6^{-1}|v-w|^{1/\alpha}
 \leqslant |\kappa_x^i(v)-\kappa_x^i(w)|
 \leqslant C_6|v-w|^\alpha,\qquad v,w\in V_x^i.
\end{equation}
\end{lemma}
\begin{proof}
The plaque identifications in \cref{lem:plaque-input} give
\[
 h_x\bigl(\Pi_x^i(\exp_x^{-1}y)\bigr)
 =\exp_x^{-1}H_{L_x}(y),\qquad y\in U_x^i.
\]
The definition of \(\pi^i\) gives \eqref{eq:coordinate-factorization}
and the asserted image. Injectivity follows from that of
\(h_x\) and \(t|_{T_{L_x}}\). Taking \(y=x\) gives
\(\kappa_x^i(0)=\pi^i(x)\).

  Each \(h_x\) is the restriction of a holonomy between
\(i\)-transversals. The inner-box choice bounds their domain
diameters and the displacements \(|v-h_x(v)|\) uniformly in
\(x\in E\). Since \(A(x)\leqslant A_0\),
\cref{lem:holonomy-input} makes \(h_x\) uniformly
\(\alpha\)-bi-H\"older. The map \(t\circ\exp_x\) is uniformly
bi-Lipschitz on \(T_x\), which proves
\eqref{eq:representative-coordinate-holder}.
\end{proof}

\subsection{Adapted partitions and transverse metrics}\label{subsec:adapted}

\subsubsection{Adapted partitions and the local strong quotient}
For a countable finite-entropy partition \(\mathcal P\), write
\(\mathcal P^+=\bigvee_{n\geqslant0}f^n\mathcal P\). This is the past-name
partition: \(y\in\mathcal P^+(x)\) means
\(\mathcal P(f^{-n}y)=\mathcal P(f^{-n}x)\) for every \(n\geqslant0\).
Its adapted refinements are
\begin{equation}\label{eq:positive-partitions}
 \eta_j=\xi_j\vee\mathcal P^+,\qquad
 \tm_x^j=m_x^{\eta_j},\qquad j=i-1,i.
\end{equation}
The additional past-name restriction in \(\eta_j(x)\) may remove
every leaf neighborhood of \(x\). Thus \(\tm_x^j=m_x^{\eta_j}\) is not
the conditional probability \(m_x^j=m_x^{\xi_j}\) used to define
\(\ldel_j,\udel_j\); their comparison is made in
Section~\ref{subsec:plaque-dimensions}.

  The partition constructions of Gan, Tong, and Yang
\cite[Lemma~2.8 and Lemmas~3.3--3.5]{GTY} and Tong
\cite[Lemmas~5.5--5.7]{Tong} give the following choice.

\Needspace{5\baselineskip}
\begin{proposition}[Adapted partitions and local quotient]\label[proposition]{prop:adapted-input}
There is \(\delta_{a,*}>0\) such that for every
\(0<\delta_a<\min\{\delta/4,\delta_{a,*}\}\) one can choose a countable
partition \(\mathcal P\) with \(H_m(\mathcal P)<\infty\), refining
\(\{S,M\setminus S\}\), \(\{E,M\setminus E\}\), and any prescribed
finite measurable partition, for which
\begin{equation}\label{eq:adapted-partition}
 \mathcal P^+(x)\subset\exp_xR_x(\delta_aA(x)^{-1})
\end{equation}
almost everywhere. The partitions in \eqref{eq:positive-partitions}
satisfy
\begin{align}
 f\eta_j&\prec\eta_j,\qquad \eta_i\prec\eta_{i-1},
 &H_m(\eta_j\mid f\eta_j)&=h_j,\label{eq:adapted-entropies}\\
 \eta_i(x)&\subset\exp_xW^i_{x,\delta}(x),
 &\eta_{i-1}(y)&=\eta_i(x)\cap\exp_xW^{i-1}_{x,2\delta}(y),
 \label{eq:local-quotient}\\
 f^{-1}\eta_{i-1}&=\eta_{i-1}\vee f^{-1}\eta_i.
 &&\label{eq:geometric-quotient}
\end{align}
The second identity in \eqref{eq:local-quotient} holds for almost every
\(x\) and every regular \(y\in\eta_i(x)\). All partition identities
are modulo null sets.
\end{proposition}

  The threshold \(\delta_{a,*}\) depends only on the fixed geometry,
\(\delta\), and \(S\). The identity
\(H_m(f^{-1}\eta_j\mid\eta_j)=h_j\) follows from the increasing
and generating properties of \(\xi_j\) and the finiteness of
\(H_m(\mathcal P)\), as in
\cite[Lemma~3.2.1]{LY1}.

  To place the strong representatives inside the coordinate domains at
returns, choose
\[
 0<\delta_a<\min\{\delta/(4C_{4}),\delta_{a,*},r_*/(2C_{4})\}.
\]
The projection bound gives, for
\(v\in\tW_x^{(i)}(0)\cap R_x(\delta_aA(x)^{-1})\),
\begin{equation}\label{eq:projection-domains}
 |\Pi_x^i(v)|'<\tfrac14\rho_\delta(x)\quad(x\in\Gamma'),\qquad
 |\Pi_x^i(v)|'<\tfrac12r_*\quad(x\in E).
\end{equation}
Fix \(\mathcal P\) given by \cref{prop:adapted-input} at this scale.
For regular \(x\) and \(y\in\eta_i(x)\), put \(p=\exp_x^{-1}y\).
Both \(p\) and \(\Pi_x^i(p)\) belong to
\(R_x(\rho_\delta(x)/4)\). The path in \eqref{eq:graph-segment} joins
them on \(\tW_x^{(i-1)}(p)\) inside
\(R_x(3\rho_\delta(x)/8)\subset R_x(2\rho_\delta(x))\).
Hence \eqref{eq:strong-plaque} gives
\begin{equation}\label{eq:representative-strong-class}
 \Pi_x^i(\exp_x^{-1}y)\in W^{i-1}_{x,2\delta}(y),
 \qquad y\in\eta_i(x).
\end{equation}
Combining this with \eqref{eq:projection-fibers} and
\eqref{eq:local-quotient}, we see that the representatives distinguish
exactly the local strong classes:
\begin{equation}\label{eq:representative-quotient}
 \Pi_x^i(\exp_x^{-1}y)=\Pi_x^i(\exp_x^{-1}z)
 \quad\Longleftrightarrow\quad
 \eta_{i-1}(y)=\eta_{i-1}(z),\qquad y,z\in\eta_i(x).
\end{equation}
Thus a strong class means an \(\eta_{i-1}\)-atom inside the given
\(\eta_i\)-atom.
For \(x\in E\), the refinements by \(\{E,M\setminus E\}\) and
\(\widehat\xi_i\), together with \(\delta_a<r_*\), give
\[
 \eta_i(x)\subset E\cap L_x\cap
 \exp_x\bigl(\tW_x^{(i)}(0)\cap R_x(r_*)\bigr)=U_x^i.
\]
Consequently \(\Pi_x^i(\exp_x^{-1}\eta_i(x))\subset V_x^i\), and
\begin{equation}\label{eq:atom-coordinate-factorization}
 \pi^i(y)=\kappa_x^i\bigl(\Pi_x^i(\exp_x^{-1}y)\bigr),
 \qquad y\in\eta_i(x).
\end{equation}
The vector \(\Pi_x^i(\exp_x^{-1}y)\) need not belong to
\(\exp_x^{-1}\eta_i(x)\); only membership in \(V_x^i\) is used.

  The quotient identity and conditional entropy monotonicity imply
\[
 h_{i-1}=H_m(f^{-1}\eta_i\mid\eta_{i-1})
 \leqslant H_m(f^{-1}\eta_i\mid\eta_i)=h_i.
\]
The choices are made in the order
\[
 (\eps;A,\theta,r_0)\longrightarrow\delta
 \longrightarrow(A_0;S,\pi,\xi_{i-1},\xi_i)
 \longrightarrow(E,r_*)\longrightarrow(\delta_a,\mathcal P).
\]

\subsubsection{Return coordinates and the transverse metric}
On \(X_E=\bigcup_{n\geqslant0}f^nE\), define the last past return
time \(\ell_0:X_E\to\Z_{\leqslant0}\) and extend
\(\pi^i|_E\) by setting
\begin{equation}\label{eq:return-coordinate}
 \ell_0(x)=\max\{\ell\leqslant0:f^\ell x\in E\},\qquad
 \tp^i(x)=\pi^i(f^{\ell_0(x)}x).
\end{equation}
Ergodicity gives \(m(X_E)=1\). For \(q\geqslant0\),
\[
 \{\ell_0=-q\}=f^qE\setminus\bigcup_{j=0}^{q-1}f^jE.
\]
The union is empty for \(q=0\). Because \(\mathcal P\) records
membership in \(E\), each set on the right belongs to
\(\sigma(\mathcal P^+)\). Hence \(\ell_0\) is constant on every
\(\eta_i\)-atom in \(X_E\). Extend \(\tp^i\) by zero outside
\(X_E\), and use the invariant full-measure set
\(\bigcap_{k\in\Z}f^kX_E\) when following orbits.

  Iterating \eqref{eq:geometric-quotient} gives, for \(q\geqslant0\),
\[
 f^{-q}\eta_{i-1}=\eta_{i-1}\vee f^{-q}\eta_i,\qquad
 \eta_{i-1}=f^q\eta_{i-1}\vee\eta_i.
\]
Fix \(x\) with \(\ell_0(x)=-q\) and \(y\in\eta_i(x)\).
The increasing property gives
\(f^{-q}y\in\eta_i(f^{-q}x)\subset E\). Applying
\eqref{eq:representative-quotient} and
\eqref{eq:atom-coordinate-factorization} in this return atom yields
\[
 \begin{split}
 y\in\eta_{i-1}(x)
 &\iff f^{-q}y\in\eta_{i-1}(f^{-q}x)\\
 &\iff\pi^i(f^{-q}y)=\pi^i(f^{-q}x)
 \iff\tp^i(y)=\tp^i(x).
 \end{split}
\]
Thus \(\eta_{i-1}(x)=\eta_i(x)\cap(\tp^i)^{-1}\{\tp^i(x)\}\).
Since the partitions and \(\tp^i\) are measurable on a standard Borel
space, this gives
\begin{equation}\label{eq:geometric-coordinate}
 \sigma(\eta_{i-1})=\sigma(\eta_i)\vee(\tp^i)^{-1}\B(\R^{d_i}).
\end{equation}
The atom equalities hold on the common regular representatives; the
sigma-algebras are completed with respect to \(m\).

  For \(y,z\in\eta_i(x)\), put
\[
 \widetilde d^i(y,z)=|\tp^i(y)-\tp^i(z)|,\qquad
 \widetilde B^i(x,\rho)=
 \{y\in\eta_i(x):\widetilde d^i(x,y)\leqslant\rho\}.
\]
The function \(\widetilde d^i\) is a pseudometric on \(\eta_i(x)\)
and a metric on the quotient \(\eta_i(x)/\eta_{i-1}\): its zero-distance
classes are exactly the \(\eta_{i-1}\)-atoms.
For these adapted partitions and return coordinates, define
\begin{equation}\label{eq:transverse-dimension-definition}
 \widetilde\gamma_i(x)=\liminf_{\rho\downarrow0}
 \frac{\log\tm_x^i(\widetilde B^i(x,\rho))}{\log\rho}.
\end{equation}
This is the lower local dimension of the transverse measure
\((\tp^i)_*\tm_x^i\) at \(\tp^i(x)\). Unlike the intrinsic dimensions
in \eqref{eq:partial-dimensions}, \(\widetilde\gamma_i(x)\) depends on
the chosen \(\eps\), adapted partitions, and return coordinates,
and need not be constant or invariant. It is measurable. Comparing radii \(\rho\) and \(2\rho\) also shows
that open and closed balls give the same local dimensions.
On the return set we have the exact identity
\begin{equation}\label{eq:return-equality}
 \tp^i=\pi^i\quad\hbox{on }E.
\end{equation}

\section{Lower bound for transverse entropy}\label{sec:counting}

  The cylinder-counting argument of \cite[Section~4]{Tong} gives
the following estimate and applies without a multiplicity
restriction.

\begin{proposition}[Lower entropy bound]\label[proposition]{lem:counting}
Under the hypotheses of \cref{thm:main},
\begin{equation}\label{eq:counting}
 h_i-h_{i-1}\geqslant\lambda_i
 \max\{\ldel_i-\ldel_{i-1},\udel_i-\udel_{i-1}\}.
\end{equation}
\end{proposition}
\begin{proof}
The cylinder estimates of \cite[Lemma~4.1]{Tong} use a
finite-entropy partition and the conditional probabilities on
the two subordinate partitions. Their geometric input,
\cite[Claim~4.2]{Tong}, is backward contraction along \(W^i\)
at the rate governed by its least exponent \(\lambda_i\).
This is supplied by \eqref{eq:strong-backward}, regardless of
its multiplicity. The proof of \cite[Section~4]{Tong} therefore
applies. The lower- and upper-dimension cases interchange an
eventual local-dimension estimate with one along a subsequence,
as explained at the end of that section. This gives both bounds
in \eqref{eq:counting} without assuming exact dimensionality.
\end{proof}

\section{Upper bound for transverse entropy}\label{sec:upper}

  Keep the geometric setup and the partitions from
Sections~\ref{sec:geometry}--\ref{sec:partitions} fixed. We prove
\[
 \min\{\alpha^{-1}\ldel_i-\ldel_{i-1},
        \alpha^{-1}\udel_i-\udel_{i-1}\}
 \geqslant \frac{\alpha}{\lambda_i+4\eps}(h_i-h_{i-1}).
\]
We first bound the transverse diameters of the atoms \(\A_n^k\)
defined below, then estimate \(\tm_x^i(\A_n(x))\). The resulting
transverse dimension bound is transferred to
\(\ldel_i,\udel_i\) in Section~\ref{subsec:plaque-dimensions}.

\subsection{Transverse cells and their iterates}\label{subsec:transverse-cells}

  Fix \(a>0\) and a bounded box \(K\subset\R^{d_i}\) containing
\(\tp^i(M)\). For \(n\geqslant1\), let \(\mathcal D_n\) be the
partition of \(K\) obtained by intersecting it with a translated grid
of half-open cubes of side \(r_n=e^{-an}\). The translation may depend
on \(n\). Define
\begin{equation}\label{eq:transported}
 \A_n=\eta_i\vee(\tp^i)^{-1}\mathcal D_n,\qquad
 \A_n^k=f^k\A_n\vee\eta_i\quad(0\leqslant k\leqslant n).
\end{equation}
We call the atoms of \(\A_n^k\) transverse cells. At time zero,
\[
 \A_n(x)=\eta_i(x)\cap(\tp^i)^{-1}\bigl(\mathcal D_n(\tp^i(x))\bigr).
\]
At time \(k\), the corresponding atom is
\begin{equation}\label{eq:transported-atom}
 \A_n^k(f^kx)=f^k(\A_n(x))\cap\eta_i(f^kx).
\end{equation}
Since \(f^k\eta_i\prec\eta_i\),
\begin{equation}\label{eq:transported-identities}
 \A_n^k=\eta_i\vee f^k\bigl((\tp^i)^{-1}\mathcal D_n\bigr),
 \qquad f^{-k}\A_n^k=\A_n\vee f^{-k}\eta_i.
\end{equation}
Thus each \(\eta_i\)-atom is split into at most
\(\#\mathcal D_n\) atoms of \(\A_n^k\). The intersection with
\(\eta_i(f^kx)\) in \eqref{eq:transported-atom} also gives
\[
 y\in f^{-k}(\A_n^k(f^kx))
 \quad\Longrightarrow\quad
 f^jy\in\eta_i(f^jx),\qquad 0\leqslant j\leqslant k,
\]
by the increasing property of \(\eta_i\). The atoms \(\A_n^k(x)\) need not
be connected.

\begin{proposition}[Transverse width along an orbit]\label[proposition]{prop:transport-geometry}
If
\begin{equation}\label{eq:scale-condition}
 a>\frac{\lambda_i+3\eps}{\alpha},
\end{equation}
then, for almost every \(x\),
\begin{equation}\label{eq:geometric-shrinking}
 \max_{0\leqslant k\leqslant n}
 \diam\tp^i\bigl(\A_n^k(f^kx)\bigr)\longrightarrow0.
\end{equation}
Diameters are taken after intersection with the full-measure set
\(X\) specified in Section~\ref{sec:partitions}. The conclusion holds
for every sequence of translated grids.
\end{proposition}
\begin{proof}
Fix \(x\in X\cap X_E\), put \(q=-\ell_0(x)\), and write \(x_j=f^jx\).
Thus \(x_{-q}\in E\), and \(q\) remains fixed as \(n\to\infty\).

  For \(y\in\A_n(x)\), the increasing property gives
\(f^{-q}y\in\eta_i(x_{-q})\subset E\). Define the projected point
\(v_{-q}(y)\in\tW_{x_{-q}}^i(0)\) and its iterates by
\[
 v_{-q}(y)=\Pi_{x_{-q}}^i(\exp_{x_{-q}}^{-1}f^{-q}y),\qquad
 v_j(y)=\tg_{x_{-q}}^{j+q}v_{-q}(y)\quad(j\geqslant-q).
\]
Since \(f^{-q}y\in U_{x_{-q}}^i\), the coordinate factorization
\eqref{eq:atom-coordinate-factorization} gives
\[
 |\kappa_{x_{-q}}^i(v_{-q}(y))-\kappa_{x_{-q}}^i(0)|
 =|\tp^i(y)-\tp^i(x)|\leqslant\sqrt{d_i}\,e^{-an}.
\]
Hence
\begin{equation}\label{eq:initial-representatives}
 |v_{-q}(y)|\leqslant C_6d_i^{\alpha/2}e^{-a\alpha n}.
\end{equation}
A radial graph curve in \(\tW_{x_{-q}}^i(0)\) joins \(0\) to
\(v_{-q}(y)\) with length at most \(2C_1|v_{-q}(y)|\). Applying
\eqref{eq:F-forward} to this curve yields
\begin{equation}\label{eq:representative-growth}
 |v_j(y)|'
 \leqslant C_7 A(x_{-q})e^{-a\alpha n+(\lambda_i+3\eps)(j+q)},
 \qquad j\geqslant-q,
\end{equation}
where \(C_7\geqslant1\) depends only on \(C_1,C_2,C_6,d_i,\alpha\),
not on \(y,n,j,a\) or the grid translation.

  Fix \(0\leqslant k\leqslant n\), take
\(z\in\A_n^k(f^kx)\), and put \(y=f^{-k}z\in\A_n(x)\).
The increasing property gives
\begin{equation}\label{eq:whole-itinerary}
 f^jy\in f^{-(k-j)}\bigl(\eta_i(f^kx)\bigr)
 \subset\eta_i(x_j),\qquad -q\leqslant j\leqslant k.
\end{equation}
In particular,
\(p_j=\exp_{x_j}^{-1}f^jy\in R_{x_j}(\delta_aA(x_j)^{-1})\).
On these boxes \(\tg_{x_j}\) agrees with \(\widetilde f_{x_j}\), so
\(p_{j+1}=\tg_{x_j}(p_j)\) for \(-q\leqslant j<k\).
Using \eqref{eq:projection-equivariance}, we obtain
\begin{equation}\label{eq:saturation}
 p_j=\tg_{x_{-q}}^{j+q}p_{-q},\qquad
 v_j(y)=\Pi_{x_j}^i(p_j).
\end{equation}
Let
\[
 J_k=\max\{-q\leqslant j\leqslant k:x_j\in E\}.
\]
Because \(\ell_0\) is constant on \(\eta_i\)-atoms,
\(\ell_0(z)=\ell_0(f^kx)=J_k-k\). Also,
\(f^{J_k}y\in\eta_i(x_{J_k})\subset U_{x_{J_k}}^i\), so
\(v_{J_k}(y)\in V_{x_{J_k}}^i\). Therefore
\begin{equation}\label{eq:cell-return-coordinate}
 \tp^i(z)=\pi^i(f^{J_k}y)=\kappa_{x_{J_k}}^i(v_{J_k}(y)).
\end{equation}
For \(z'\in\A_n^k(f^kx)\), put \(y'=f^{-k}z'\).
The upper bound in \eqref{eq:representative-coordinate-holder} gives
\[
 |\tp^i(z)-\tp^i(z')|
 \leqslant C_6|v_{J_k}(y)-v_{J_k}(y')|^\alpha.
\]
Since \(J_k\leqslant k\leqslant n\), \eqref{eq:representative-growth}
implies
\begin{equation}\label{eq:return-diameter}
 \diam\tp^i\bigl(\A_n^k(f^kx)\bigr)
 \leqslant C_xe^{-\alpha(a\alpha-\lambda_i-3\eps)n}.
\end{equation}
Here \(C_x\) may depend on \(x\), including its past return time
\(q\), but is independent of \(n,k,a\) and the grid translation.
Since \(a\alpha-\lambda_i-3\eps>0\),
\eqref{eq:return-diameter} proves \eqref{eq:geometric-shrinking}.
\end{proof}

\subsection{Mass of a transverse cell}\label{sec:information}

  For measurable partitions \(\mathcal U,\mathcal V\), write
\begin{equation}\label{eq:information-definition}
 I_m(\mathcal U\mid\mathcal V)(x)
 =-\log m_x^{\mathcal V}(\mathcal U(x)),\qquad
 H_m(\mathcal U\mid\mathcal V)=\int I_m(\mathcal U\mid\mathcal V)\,dm,
\end{equation}
with \(-\log0=+\infty\); we use the conditional chain rule and
continuity under increasing conditioning as in \cite{Rokhlin}.

  We estimate \(\tm_x^i(\A_n(x))\) by comparing the probabilities of
\((f^{-n}\eta_i)(x)\) under \(\tm_x^i=m_x^{\eta_i}\) and
\(m_x^{\A_n}\). For the second probability, we assign weights
along successive cylinder refinements. The following elementary
lemma converts a bound on their total weight into a lower bound
for the true cylinder masses outside a small exceptional set.

\Needspace{7\baselineskip}
\begin{lemma}[Comparison weights]\label[lemma]{lem:comparison-weights}
Let \(\mu\) be a probability measure and \(\beta\) a countable
measurable partition. If \(w(B)\geqslant0\), \(B\in\beta\), and
\(\sum_{B\in\beta}w(B)\leqslant1\), then, for every \(t>0\),
\[
 \mu\{x:\mu(\beta(x))<t\,w(\beta(x))\}\leqslant t.
\]
\end{lemma}
\begin{proof}
Sum \(\mu(B)<t\,w(B)\) over the atoms satisfying this inequality.
\end{proof}

\begin{theorem}[Mass of transverse cells]\label{thm:information}
For every \(a>0\) satisfying \eqref{eq:scale-condition} and every
sequence of translated grids in \eqref{eq:transported},
\begin{equation}\label{eq:information-bound}
 \liminf_{n\to\infty}-\frac1n
 \log\tm_x^i\bigl(\A_n(x)\bigr)\geqslant h_i-h_{i-1}
 \quad\text{for almost every }x.
\end{equation}
\end{theorem}
\begin{proof}
Fix \(a>0\) satisfying \eqref{eq:scale-condition} and the grid
sequence. Put
\[
 I_0=I_m(f^{-1}\eta_i\mid\eta_i),
 \qquad I_0\in L^1(m),\qquad \int I_0\,dm=h_i.
\]
Using the conditional chain rule and
\eqref{eq:conditional-equivariance}, we compute the
\(\tm_x^i\)-mass of \(\A_n(x)\cap(f^{-n}\eta_i)(x)\) in two
ways and obtain
\begin{equation}\label{eq:telescope}
 \begin{aligned}
 \tm_x^i(\A_n(x))\,
   m_x^{\A_n}((f^{-n}\eta_i)(x))
 &=\tm_x^i((f^{-n}\eta_i)(x))\,
   m_x^{f^{-n}\eta_i}(\A_n(x))\\
 &=\exp\Bigl(-\sum_{k=0}^{n-1}I_0(f^kx)\Bigr)
   \tm_{f^nx}^i(\A_n^n(f^nx)).
 \end{aligned}
\end{equation}
All factors are positive almost everywhere: both \(\A_n\) and
\(\A_n^n\) are finite refinements of \(\eta_i\), and
\(H_m(f^{-n}\eta_i\mid\A_n)\leqslant nh_i<\infty\).
Since the last factor is at most one,
\begin{equation}\label{eq:mass-reduction}
 -\log\tm_x^i(\A_n(x))
 \geqslant\sum_{k=0}^{n-1}I_0(f^kx)
       +\log m_x^{\A_n}((f^{-n}\eta_i)(x)).
\end{equation}
The first sum divided by \(n\) converges to \(h_i\) by Birkhoff's
theorem. It therefore suffices to prove
\[
 \limsup_{n\to\infty}-\frac1n
 \log m_x^{\A_n}((f^{-n}\eta_i)(x))\leqslant h_{i-1}
 \quad\text{almost everywhere}.
\]

  Choose nested grids \(\Q_\ell\), \(\ell\geqslant1\), of half-open
cubes of side \(2^{-\ell}\) in \(\R^{d_i}\), with a common
translation such that
\[
 m((\tp^i)^{-1}\partial\Q_\ell)=0
 \qquad\text{for every }\ell.
\]
Here \(\partial\Q_\ell\) is the union of the grid hyperplanes.
Such a translation exists because each coordinate marginal of
\((\tp^i)_*m\) has at most countably many atoms. Define
\[
 \C_\ell=\eta_i\vee(\tp^i)^{-1}\Q_\ell,
 \qquad I_\ell=I_m(f^{-1}\eta_i\mid\C_\ell).
\]
By \eqref{eq:geometric-coordinate},
\(\bigvee_{\ell\geqslant1}\C_\ell=\eta_{i-1}\). Continuity of
conditional entropy under increasing conditioning, using
\(H_m(f^{-1}\eta_i\mid\eta_i)=h_i<\infty\), gives
\begin{equation}\label{eq:entropy-convergence}
 \int I_\ell\,dm
 =H_m(f^{-1}\eta_i\mid\C_\ell)
 \downarrow H_m(f^{-1}\eta_i\mid\eta_{i-1})
 =h_{i-1},
\end{equation}
where the last equality follows from
\eqref{eq:geometric-quotient} and \eqref{eq:adapted-entropies}.
In particular, \(I_\ell\in L^1(m)\). Fix \(\ell\) until the
final limit.

  For \(n\geqslant1\) and \(0\leqslant k<n\), define
\begin{equation}\label{eq:predictor}
 q_n^k(x)=
 \begin{cases}
 m_x^{\C_\ell}((f^{-1}\eta_i)(x)),
   &m_x^{\A_n^k}(\C_\ell(x))=1,\\
 \tm_x^i((f^{-1}\eta_i)(x)),
   &\text{otherwise}.
 \end{cases}
\end{equation}
For almost every atom \(D\) of \(\A_n^k\), write \(m_D\) for
its conditional probability. If \(m_D(C)=1\) for some
\(C\in\C_\ell\), let \(\nu_D\) be the \(\C_\ell\)-conditional
probability on \(C\); otherwise, let it be the \(\eta_i\)-conditional
probability on the atom containing \(D\). Both \(\A_n^k\) and
\(\C_\ell\) finitely refine \(\eta_i\). With the above choice of
branch, their conditional probabilities satisfy
\(m_D=\nu_D|_D/\nu_D(D)\), with \(\nu_D(D)>0\); in particular,
\(m_D\ll\nu_D\). The definition of \(q_n^k\) uses the fixed
measure \(\nu_D\) itself, without restriction to \(D\):
\[
 q_n^k(y)=\nu_D(F)
 \quad\text{for }m_D\text{-almost every }y\in D\cap F,
 \qquad F\in f^{-1}\eta_i.
\]
Each \(q_n^k\) is measurable. Positivity
almost everywhere follows from \(I_\ell,I_0\in L^1(m)\).

  We first show that the products of these one-step probabilities
are valid comparison weights. Fix \(n\). Since
\[
 H_m(f^{-k}\eta_i\mid\A_n)\leqslant kh_i<\infty,
 \qquad 0\leqslant k\leqslant n,
\]
for almost every \(\A_n\)-atom \(A\), with conditional probability
\(m_A\), the intersections \(A\cap F'\), \(F'\in f^{-k}\eta_i\),
of positive \(m_A\)-mass form a countable partition modulo
\(m_A\) for each \(k\). Fix such an \(A\) on which the conditional
identities below also hold. The increasing property of
\(\eta_i\) makes these partitions the successive levels of a
tree with root \(A\).

  For a positive-mass atom \(B\) at depth \(k<n\), put
\(D=f^kB\in\A_n^k\), as in \eqref{eq:transported-identities}.
Its children are \(B\cap f^{-k}F\), where
\(F\in f^{-1}\eta_i\). Disintegration and
\eqref{eq:conditional-equivariance} give
\[
 \frac{m_A(B\cap f^{-k}F)}{m_A(B)}=m_D(F).
\]
Thus its positive-mass children correspond exactly to those
\(F\) with \(m_D(F)>0\). Assign weight \(1\) to the root
\(A\), and give each such child its parent's weight multiplied
by \(\nu_D(F)\). Since
\[
 \sum_{\substack{F\in f^{-1}\eta_i\\m_D(F)>0}}\nu_D(F)
 \leqslant1,
\]
the sum of the children's weights does not exceed the parent's
weight. Induction over the depth therefore bounds the total
weight at every level by \(1\).

  On a positive-mass terminal atom
\(B=A\cap F'\), \(F'\in f^{-n}\eta_i\), each factor
\(q_n^k\circ f^k\) equals the corresponding branch multiplier
\(m_A\)-almost everywhere. The terminal weight is therefore
\(\prod_{k=0}^{n-1}q_n^k(f^kx_B)\), with \(x_B\in B\) chosen
where these identities hold. Hence
\[
 \sum_B\prod_{k=0}^{n-1}q_n^k(f^kx_B)\leqslant1,
\]
where the sum is over all terminal atoms in \(A\) of positive
\(m_A\)-mass.
The true mass of each terminal atom satisfies
\(m_A(B)=m_x^{\A_n}((f^{-n}\eta_i)(x))\) for
\(m_A\)-almost every \(x\in B\). Assigning weight zero to the
null remainder, apply \cref{lem:comparison-weights} with
\(\mu=m_A\) and \(t=n^{-2}\), and integrate with respect to the
quotient measure on \(\A_n\). This gives
\[
 m\left\{x:
 m_x^{\A_n}((f^{-n}\eta_i)(x))
 <n^{-2}\prod_{k=0}^{n-1}q_n^k(f^kx)
 \right\}\leqslant n^{-2}.
\]
Borel--Cantelli yields, for \(m\)-almost every \(x\) and all
sufficiently large \(n\),
\begin{equation}\label{eq:likelihood-conclusion}
 -\log m_x^{\A_n}((f^{-n}\eta_i)(x))
 \leqslant
 \sum_{k=0}^{n-1}-\log q_n^k(f^kx)+2\log n.
\end{equation}

  It remains to bound the information of the comparison weights.
For \(\rho>0\), put
\[
 V_{\ell,\rho}
 =\{x:\dist(\tp^i(x),\partial\Q_\ell)\leqslant\rho\}.
\]
By \cref{prop:transport-geometry}, for almost every \(x\) and all
large \(n\),
\[
 \diam\tp^i(\A_n^k(f^kx))<\rho,
 \qquad 0\leqslant k<n.
\]
If \(f^kx\notin V_{\ell,\rho}\), this transverse image lies in
\(\Q_\ell(\tp^i(f^kx))\), so the first case of
\eqref{eq:predictor} applies. At the remaining times, either case
may apply. Since \(I_\ell,I_0\geqslant0\), we obtain
\begin{equation}\label{eq:predictor-information}
 -\log q_n^k(f^kx)
 \leqslant
 \bigl(I_\ell+\one_{V_{\ell,\rho}}I_0\bigr)(f^kx),
 \qquad 0\leqslant k<n,
\end{equation}
for all large \(n\). Birkhoff's theorem, with \(\ell,\rho\) fixed,
gives
\[
 \limsup_{n\to\infty}
 \frac1n\sum_{k=0}^{n-1}-\log q_n^k(f^kx)
 \leqslant\int I_\ell\,dm+\int_{V_{\ell,\rho}}I_0\,dm.
\]
Together with \eqref{eq:likelihood-conclusion},
\eqref{eq:mass-reduction}, and Birkhoff's theorem for \(I_0\),
this implies
\begin{equation}\label{eq:predictor-average}
 \liminf_{n\to\infty}-\frac1n\log\tm_x^i(\A_n(x))
 \geqslant h_i-\int I_\ell\,dm
              -\int_{V_{\ell,\rho}}I_0\,dm.
\end{equation}
Since \(I_0\in L^1(m)\) and the grid boundary has zero measure,
\(\int_{V_{\ell,\rho}}I_0\,dm\to0\) as \(\rho\downarrow0\).
Letting first \(\rho\downarrow0\) and then \(\ell\to\infty\),
and using \eqref{eq:entropy-convergence}, we obtain the claimed
bound.
Both limits are taken along countable sequences on a common
full-measure set, after the \(n\to\infty\) limit.
\end{proof}

  Integrating \eqref{eq:information-bound}, applying Fatou's lemma,
and using \(H_m(\A_n\mid\eta_i)\leqslant\log\#\mathcal D_n\), we obtain
\begin{equation}\label{eq:transverse-cell-count}
 h_i-h_{i-1}
 \leqslant\liminf_{n\to\infty}\frac1n H_m(\A_n\mid\eta_i)
 \leqslant\liminf_{n\to\infty}\frac1n\log\#\mathcal D_n.
\end{equation}

\begin{remark}\label{rem:cell-mass-inputs}
For the fixed ergodic system \((f,m)\), the proof above gives
\eqref{eq:information-bound} and \eqref{eq:transverse-cell-count}
for any increasing measurable partitions \(\eta_i\prec\eta_{i-1}\)
satisfying \eqref{eq:geometric-quotient} and
\[
 H_m(f^{-1}\eta_j\mid\eta_j)=h_j<\infty,
 \qquad j=i-1,i,
\]
provided that a bounded Borel coordinate \(\tp^i:M\to\R^{d_i}\),
with arbitrary finite \(d_i\), satisfies
\eqref{eq:geometric-coordinate} and its finite grid refinements
\eqref{eq:transported} satisfy
\eqref{eq:geometric-shrinking}. The scale condition
\eqref{eq:scale-condition} is used only to establish this shrinking
property, not in the mass comparison.
\end{remark}

\subsection{From cells to transverse balls}\label{subsec:transverse-balls}

  We pass from the masses of \(\A_n(x)\) to the masses of
\(\widetilde B^i(x,r)\). The required containment is obtained by
shifting each coordinate grid by \(0,r_n/3,2r_n/3\).

\begin{corollary}[Transverse dimension]\label[corollary]{thm:transverse-estimate}
For every sufficiently small fixed \(\eps>0\),
\begin{equation}\label{eq:transverse-estimate}
 \widetilde\gamma_i(x)\geqslant
 \frac{\alpha}{\lambda_i+4\eps}(h_i-h_{i-1})
 \quad\text{for almost every }x.
\end{equation}
\end{corollary}
\begin{proof}
Take \(a=(\lambda_i+4\eps)/\alpha\) and \(r_n=e^{-an}\).
For \(v\in\{0,1,2\}^{d_i}\), translate the grid of side \(r_n\) by
\(r_nv/3\), and write \(Q_{n,v}(z)\) for its full cube containing
\(z\in\R^{d_i}\). In each coordinate, one of the three shifts places
\(z\) at distance at least \(r_n/6\) from the boundary. Hence, for
some \(v=v(n,z)\),
\[
 B(z,r_n/10)\subset Q_{n,v}(z).
\]
Let \(\A_{n,v}\) be the partition defined by \eqref{eq:transported}
using the restriction of this shifted grid to \(K\).
Fix \(x\) in the common full-measure set of \cref{thm:information}
for these \(3^{d_i}\) grid sequences, and put
\(\nu_x=(\tp^i)_*\tm_x^i\), \(z=\tp^i(x)\).
For every \(\zeta>0\), the theorem gives, for all large \(n\) and
every \(v\),
\[
 \nu_x(Q_{n,v}(z))=\tm_x^i(\A_{n,v}(x))
 \leqslant e^{-n(h_i-h_{i-1}-\zeta)}.
\]
The equality holds because \(\nu_x(K)=1\). Choosing \(v(n,z)\) gives
the same upper bound for \(\nu_x(B(z,r_n/10))\). If
\(r_{n+1}/10<r\leqslant r_n/10\), monotonicity yields
\[
 \frac{-\log\nu_x(B(z,r))}{-\log r}
 \geqslant
 \frac{-\log\nu_x(B(z,r_n/10))}{a(n+1)+\log10}.
\]
Let \(n\to\infty\) and then \(\zeta\downarrow0\). This proves
\(\widetilde\gamma_i(x)\geqslant(h_i-h_{i-1})/a\).
\end{proof}

\subsection{Transfer to intrinsic dimensions}\label{subsec:plaque-dimensions}

  We use \cite[Lemmas~5.1, 5.3, and~5.4]{Tong}; the latter two
are the dimension inequalities in
\cite[Lemmas~11.3.1--11.3.2]{LY2}. As in
\cite[Section~5.2.2]{Tong}, we apply them to a return plaque.
Write \(\ld_\omega\) and \(\ud_\omega\) for the lower and upper
local dimensions, respectively, of a probability measure
\(\omega\).

\begin{proposition}[Upper entropy bound]\label[proposition]{prop:restriction-transfer}
With \(\eps\) and \(\alpha\) as above, we have
\begin{equation}\label{eq:transfer-bound}
 \begin{aligned}
 \alpha^{-1}\ldel_i-\ldel_{i-1}
 &\geqslant \frac{\alpha(h_i-h_{i-1})}{\lambda_i+4\eps},\\
 \alpha^{-1}\udel_i-\udel_{i-1}
 &\geqslant \frac{\alpha(h_i-h_{i-1})}{\lambda_i+4\eps}.
 \end{aligned}
\end{equation}
\end{proposition}
\begin{proof}
Disintegrate \(m|_E/m(E)\) over \(\widehat\xi_i|_E\). For almost
every \(W^i\)-plaque \(L\) in the resulting quotient, let \(\mu\)
be the conditional probability on \(E\cap L\), viewed as a
probability on \(L\), and write
\[
 \nu=\pi_*\mu,\qquad \bar\nu=(\pi^i)_*\mu,
 \qquad \nu(du,dt)=\nu_t(du)\,\bar\nu(dt).
\]
For \(j=i-1,i\), the refinement from \(\widehat\xi_j\) to
\(\xi_j\) is countable modulo plaque probabilities. Indeed,
before restriction to the common full-measure set, distinct
subordinate atoms in a fixed leaf contain disjoint relative
leaf neighborhoods, and the leaf is second countable.
The conditional probabilities on subordinate atoms are
therefore normalized restrictions of the corresponding plaque
probabilities. Differentiation preserves both local dimensions
almost everywhere under these restrictions and under restriction
to \(E\). Together with the \(\alpha\)-H\"older bound for
\(\pi|_L\) and its bi-Lipschitz restriction to each strong plaque,
this gives
\[
 \begin{aligned}
 \ld_\nu(u,t)&\leqslant\alpha^{-1}\ldel_i,
 &\ud_\nu(u,t)&\leqslant\alpha^{-1}\udel_i,\\
 \ld_{\nu_t}(u)&=\ldel_{i-1},
 &\ud_{\nu_t}(u)&=\udel_{i-1},
 \end{aligned}
\]
for \(\nu\)-almost every \((u,t)\).

  Since \(\eta_i\) refines \(\widehat\xi_i\) and \(E\) is a
union of \(\eta_i\)-atoms, iterated disintegration gives
\[
 \mu=\int_L\tm_y^i\,d\mu(y),
 \qquad
 \bar\nu=\int_L(\pi^i)_*\tm_y^i\,d\mu(y).
\]
On \(E\), we have \(\tp^i=\pi^i\), and \(\tm_y^i\) is almost
everywhere constant on each typical \(\eta_i\)-atom. Thus, for
\(\mu\)-almost every \(y\), \cref{thm:transverse-estimate} gives
\[
 \ld_{(\pi^i)_*\tm_y^i}(t)
 \geqslant \frac{\alpha(h_i-h_{i-1})}{\lambda_i+4\eps}
\]
for \((\pi^i)_*\tm_y^i\)-almost every \(t\).
The mixture inequality \cite[Lemma~5.4]{Tong} therefore yields
\[
 \ld_{\bar\nu}(t)
 \geqslant \frac{\alpha(h_i-h_{i-1})}{\lambda_i+4\eps}
 \qquad\text{for }\bar\nu\text{-almost every }t.
\]

  The lower product inequality \cite[Lemma~5.3]{Tong}, with
fiber bound \(\ldel_{i-1}\), now gives the first inequality in
\eqref{eq:transfer-bound}. For the second, apply
\cite[Lemma~5.1]{Tong} with the constant upper bound
\(\alpha^{-1}\udel_i\) for \(\ud_\nu\), obtaining
\[
 \ud_{\nu_t}(u)+\ld_{\bar\nu}(t)
 \leqslant\alpha^{-1}\udel_i
\]
for \(\nu\)-almost every \((u,t)\). Combining this with
\(\ud_{\nu_t}(u)=\udel_{i-1}\) and the lower bound for
\(\ld_{\bar\nu}\) proves the second inequality.
\end{proof}

\section{Entropy formulas}\label{sec:consequences}

  We combine \eqref{eq:counting} and \eqref{eq:transfer-bound}, then
apply the same cell-mass estimate to the nonpositive-exponent bundle.

\subsection{The transverse entropy formula}
\begin{proof}[Proof of \cref{thm:main}]
For fixed \(\eps\), take \(a=(\lambda_i+4\eps)/\alpha\), as in
\cref{thm:transverse-estimate}. The number of cubes of side
\(e^{-an}\) meeting the bounded coordinate box \(K\) is at most
\(C_K e^{and_i}\), uniformly in their translation. Hence
\(\limsup_n n^{-1}\log\#\mathcal D_n\leqslant ad_i\), and
\eqref{eq:transverse-cell-count} gives
\[
 h_i-h_{i-1}\leqslant\frac{\lambda_i+4\eps}{\alpha}\,d_i.
\]
The partial entropies and intrinsic dimensions do not depend on the
auxiliary choices. We may therefore combine this estimate with
\cref{prop:restriction-transfer,lem:counting} and, after taking the
small-scale limits at fixed \(\eps\), let \(\eps\downarrow0\).
Since \(\alpha\to1\), we obtain
\[
 \begin{aligned}
 \max\{\ldel_i-\ldel_{i-1},\udel_i-\udel_{i-1}\}
 &\leqslant\frac{h_i-h_{i-1}}{\lambda_i}\\
 &\leqslant\min\{\ldel_i-\ldel_{i-1},\udel_i-\udel_{i-1},d_i\}.
 \end{aligned}
\]
Since \(h_i\geqslant h_{i-1}\), this proves \eqref{eq:main-formula}.
\end{proof}

\subsection{The unstable entropy formula}\label{subsec:unstable-proof}

  We now assume only that \(E^u\oplus E^{cs}\) is dominated.
We use forward length bounds with arbitrarily small exponential
rates along the fake \(cs\)-leaves and a positive H\"older
exponent independent of these rates. Unlike the transverse
dimension comparison, this argument does not require the
H\"older exponent to approach one.

\subsubsection{Fake foliations and local unstable manifolds}
Assume first that both bundles are nonzero. Put
\(D=\dim E^{cs}\), and let \(\lambda^u>0\) be the least
positive exponent. We repeat the geometric and partition
constructions for this splitting, choosing all auxiliary parameters
anew. Fix
\(0<\eps<\lambda^u/14\) sufficiently small and use the
cutoff and graph-transform constructions
of \cite[Section~2.1]{GTY} and \cite[Section~7.2.1]{Tong} to
obtain \(\tg_x\) and the fake foliations
\(\tW_x^u,\tW_x^{cs}\), whose leaves are complete graphs of
small slope over \(E_x^u,E_x^{cs}\), respectively. Use the
corresponding box norm
\[
 |v|'=\max\{|v_u|,|v_{cs}|\},\qquad
 R_x(r)=\{v:|v|'\leqslant r\}.
\]
The largest exponent on \(E^{cs}\) is nonpositive. The
finite-block estimates of \cite[Section~8.1]{ABC}, used in
\cite[Lemmas~2.1--2.2]{GTY}, apply to \(E^{cs}\) forward and
\(E^u\) backward. Choose the cutoff radius and graph slopes so
that the length argument of \cite[Remark~3]{GTY} gives a Borel
function \(A\geqslant1\), with
\(A(f^{\pm1}x)\leqslant e^\eps A(x)\), and a constant
\(C_8\geqslant1\) such that
\begin{align}
 \operatorname{length}(\tg_x^n\gamma)
 &\leqslant C_8A(x)e^{3\eps n}\operatorname{length}(\gamma),
 &&\gamma\subset\tW_x^{cs}(v),\label{eq:cs-growth}\\
 \operatorname{length}(\tg_x^{-n}\gamma)
 &\leqslant C_8A(x)e^{-(\lambda^u-3\eps)n}
                 \operatorname{length}(\gamma),
 &&\gamma\subset\tW_x^u(v).\label{eq:u-backward}
\end{align}
These bounds hold for all \(n\geqslant0\), all rectifiable
curves in the indicated leaves, and all base points in a
common invariant set of full measure. The constant \(C_8\)
may depend on \(\eps\), but not on \(x,n,v,\gamma\).

  Projection along \(\tW_x^u\) onto the standard fake
\(cs\)-leaf is defined on the whole tangent space by
\[
 \Pi_x^{cs}:T_xM\longrightarrow\tW_x^{cs}(0),\qquad
 \{\Pi_x^{cs}(v)\}=\tW_x^u(v)\cap\tW_x^{cs}(0).
\]
The small graph slopes give uniqueness and, as in
\eqref{eq:projection-bound}--\eqref{eq:projection-equivariance},
\begin{equation}\label{eq:cs-projection}
 |\Pi_x^{cs}(v)|'\leqslant2|v|',\qquad
 \Pi_{fx}^{cs}\circ\tg_x=\tg_x\circ\Pi_x^{cs}.
\end{equation}
A \(cs\)-transversal is an open subset of a complete \(C^1\)
graph over \(E_x^{cs}\), relative to \(E_x^u\), with sufficiently
small slope. By the proof of \cite[Lemma~2.4]{GTY}, holonomies
along \(\tW_x^u\) between such transversals are
\(\alpha_0\)-bi-H\"older for an exponent \(\alpha_0\in(0,1)\)
independent of sufficiently small \(\eps\). For fixed \(\eps\),
the constants are uniform when \(A(x)\), the domain diameter,
and the holonomy displacement are bounded. In that proof, the
backward center estimate is replaced by a Lipschitz bound for
\(\tg_x^{-1}\) uniform in \(x\) and sufficiently small \(\eps\).
Domination keeps the backward transverse graphs
in small \(cs\)-cones, so the same forward length comparison
and \eqref{eq:u-backward} apply.

  Choose \(\delta>0\) sufficiently small, put
\(\rho_\delta(x)=\delta A(x)^{-1}\), and define
\begin{equation}\label{eq:past-trapped}
 S^-_\delta(x)=
 \left\{v\in T_xM:
 \tg_x^{-n}v\in R_{f^{-n}x}(\rho_\delta(f^{-n}x))
 \text{ for every }n\geqslant0\right\}.
\end{equation}
The set \(S^-_\delta(x)\) consists of points whose backward
iterates remain in the agreement charts. The same backward-orbit
condition defines \(S^{cu}_\delta(x)\) in
\cite[Section~2.3]{GTY}. Here we use it for the two-block
splitting, without the fake \(cu\)-leaf containment of
\cite[Lemma~2.7]{GTY}.

  For regular \(y\in\exp_xS^-_\delta(x)\), let
\(W^u_{x,2\delta}(y)\) be the inverse exponential image of
the component through \(y\) of
\(W^u(y)\cap\exp_xR_x(2\rho_\delta(x))\).
The local argument of \cite[Lemma~2.5]{GTY}, using
\eqref{eq:u-backward} and the backward-orbit condition, identifies
this plaque with the fake unstable leaf. Together with the
recurrence argument of \cite[Lemma~2.6\textup{(B2)}]{GTY}, it
gives, for almost every \(x\) and every such regular \(y\),
\begin{gather}
 W^u_{x,2\delta}(y)
 =\tW_x^u(\exp_x^{-1}y)\cap R_x(2\rho_\delta(x)),
 \label{eq:two-block-plaque}\\*
 S^-_\delta(x)\cap\exp_x^{-1}W^u(y)
 \subset W^u_{x,2\delta}(y).
 \label{eq:two-block-local-class}
\end{gather}
The one-step compatibility in \cite[Lemma~2.6\textup{(B1)}]{GTY}
holds when \(y\) and \(fy\) satisfy the corresponding
backward-orbit conditions at \(x\) and \(fx\).

\subsubsection{Return coordinates and the entropy estimate}
\begin{proof}[Proof of \cref{thm:unstable}]
If \(E^u=0\) or \(E^{cs}=0\), Ruelle's inequality for \(f\)
or \(f^{-1}\) gives \(h_m(f)=0\), and hence
\(h_m^u(f)=0\). Otherwise use the fake foliations, projections, and
estimates above.
Let \(X\) be a common invariant full-measure set on which the
geometric and partition identities used below hold.

  The constructions of \cite[Sections~3.1, 3.3, and~5.1]{GTY}
give a local unstable lamination box \(S\), an increasing
generating partition \(\xi^u\) subordinate to \(W^u\), a
positive-measure return set \(E\subset S\cap X\) on which
\(A\) is bounded, and a fixed \(D\)-dimensional transversal
\(T=\tau(B^D)\). Here \(B^D\subset\R^D\) is a bounded disk,
and \(\tau\) is a bi-Lipschitz \(C^1\) parametrization. Every
local \(W^u\)-plaque of \(S\) meeting \(E\) meets \(T\) in
one point; let \(\phi_0(y)\in B^D\) be its coordinate for
\(y\in E\). Choose the inner box and a uniform radius \(r_*>0\)
so that \(\exp_x^{-1}T\) is a \(cs\)-transversal for \(x\in E\).
For \(y\in E\cap\exp_x(S^-_\delta(x)\cap R_x(r_*))\), the local
\(W^u\)-plaque through \(y\) reaches \(T\) inside the chart of
\eqref{eq:two-block-plaque}. Its representative
\(\Pi_x^{cs}(\exp_x^{-1}y)\) reaches the same intersection
by fake unstable holonomy, with uniformly bounded displacement.
These choices are made before the adapted partition.

  Choose a finite partition \(\mathcal P_0\) with
\(h_m(f,\mathcal P_0)\geqslant h_m(f)-\eps\).
The proof of \cite[Lemma~2.8]{GTY} gives a finite-entropy
partition \(\mathcal P\) refining \(\mathcal P_0\),
\(\{S,M\setminus S\}\), and \(\{E,M\setminus E\}\), such that
\[
 \mathcal P^+(x)\subset\exp_xS^-_{\delta_a}(x),\qquad
 0<\delta_a<\min\{\delta/4,r_*/4\}.
\]
Choose \(\delta_a\) also small enough, as in
\cite[Section~3.2]{GTY}, that points in the same
\(\mathcal P^+\)-atom and the same local \(W^u\)-plaque in the
base chart lie in one plaque of \(S\) at each past return to
\(S\). Put
\(\eta_2=\mathcal P^+\) and
\(\eta_1=\xi^u\vee\mathcal P^+\).
By \eqref{eq:two-block-plaque}--\eqref{eq:two-block-local-class}
and the proofs of \cite[Lemmas~3.3--3.5]{GTY},
\begin{gather}
 \eta_2\prec\eta_1,\qquad f\eta_j\prec\eta_j\quad(j=1,2),\qquad
 f^{-1}\eta_1=\eta_1\vee f^{-1}\eta_2,
 \label{eq:unstable-quotient}\\
 H_m(f^{-1}\eta_2\mid\eta_2)=h_m(f,\mathcal P),\qquad
 H_m(f^{-1}\eta_1\mid\eta_1)=h_m^u(f).
 \label{eq:unstable-entropies}
\end{gather}
Moreover, the \(\eta_1\)-subatoms of \(\eta_2(x)\) are exactly
its intersections with the local \(W^u\)-plaques in the
\(x\)-chart.

  For \(x\in E\), the atom \(\eta_2(x)\) is contained in
\(E\cap\exp_x(S^-_\delta(x)\cap R_x(r_*/4))\), so the preceding
holonomy comparison applies. Set
\(V_x^{cs}=\Pi_x^{cs}(\exp_x^{-1}\eta_2(x))\).
Holonomy along \(\tW_x^u\) from the standard fake \(cs\)-leaf
to \(\exp_x^{-1}T\), followed by \(\tau^{-1}\circ\exp_x\),
defines a bijection
\[
 \kappa_x^{cs}:V_x^{cs}\longrightarrow\phi_0(\eta_2(x)),\qquad
 \phi_0(y)=\kappa_x^{cs}\bigl(\Pi_x^{cs}(\exp_x^{-1}y)\bigr),
 \quad y\in\eta_2(x).
\]
In particular, \(\kappa_x^{cs}(0)=\phi_0(x)\).
The plaque identification ensures that both sides use the
same local \(W^u\)-class; equivalently, the fibers of
\(\phi_0|_{\eta_2(x)}\) are the \(\eta_1\)-subatoms.
The bound \eqref{eq:cs-projection} places \(V_x^{cs}\) in
\(R_x(r_*/2)\), within the chosen holonomy domains.
The maps \(\kappa_x^{cs}\) and their inverses are
\(\alpha_0\)-H\"older with a common constant \(C_9\geqslant1\).
This constant may depend on \(\eps\) and the return box,
but not on \(x\in E\).

  Extend \(\phi_0\) by past returns: on
\(X_E=\bigcup_{n\geqslant0}f^nE\), set
\[
 \ell_0(x)=\max\{\ell\leqslant0:f^\ell x\in E\},\qquad
 \phi(x)=\phi_0(f^{\ell_0(x)}x),
\]
and put \(\phi=0\) off \(X_E\). Ergodicity gives \(m(X_E)=1\).
Since \(\mathcal P\) records membership in \(E\), \(\ell_0\)
is constant on each \(\eta_2\)-atom. Iterating
\eqref{eq:unstable-quotient}, as in
Section~\ref{subsec:adapted}, shows that \(\phi\) is a bounded
Borel coordinate with
\begin{equation}\label{eq:cs-coordinate-quotient}
 \sigma(\eta_1)=\sigma(\eta_2)\vee\phi^{-1}\B(\R^D).
\end{equation}
We henceforth restrict \(X\) to
\(\bigcap_{k\in\Z}f^kX_E\).

  Take \(a=4\eps/\alpha_0\). Partition a bounded box containing
\(\phi(M)\) into cubes of side \(e^{-an}\), denoting the
resulting finite partition by \(\mathcal D_n\), and set
\[
 \mathcal U_n=\eta_2\vee\phi^{-1}\mathcal D_n,\qquad
 \mathcal U_n^k=f^k\mathcal U_n\vee\eta_2
 \quad(0\leqslant k\leqslant n).
\]
Apply the representative argument of \cref{prop:transport-geometry}
with \(\Pi^{cs},\kappa^{cs}\) and exponent \(\alpha_0\).
For \(q=-\ell_0(x)\), the initial representatives at
\(f^{-q}x\) have size at most \(C_xe^{-a\alpha_0n}\), and
\eqref{eq:cs-growth} bounds their forward growth. For
\(z\in\mathcal U_n^k(f^kx)\), the point \(y=f^{-k}z\)
satisfies \(f^jy\in\eta_2(f^jx)\) for \(-q\leqslant j\leqslant k\),
as in \eqref{eq:whole-itinerary}. Hence the actual orbit stays
in the required charts, and \eqref{eq:cs-projection} identifies
its projected iterates. At the last return before time \(k\),
the upper H\"older bound for \(\kappa^{cs}\) gives
\begin{equation}\label{eq:cs-shrinking}
 \max_{0\leqslant k\leqslant n}
 \diam\phi\bigl(\mathcal U_n^k(f^kx)\bigr)
 \leqslant C_xe^{-\alpha_0(a\alpha_0-3\eps)n}
 =C_xe^{-\alpha_0\eps n}\longrightarrow0.
\end{equation}
Here \(C_x\) may depend on the past return time \(q\), but not
on \(n,k\); all diameters are taken after intersection with \(X\).

  We can now apply \cref{rem:cell-mass-inputs} with
\((\eta_i,\eta_{i-1},\tp^i,d_i)=(\eta_2,\eta_1,\phi,D)\)
and \((h_i,h_{i-1})=(h_m(f,\mathcal P),h_m^u(f))\).
Since \(\#\mathcal D_n\leqslant C_\eps e^{anD}\),
\eqref{eq:transverse-cell-count} gives
\[
 h_m(f,\mathcal P)-h_m^u(f)\leqslant aD.
\]
Consequently,
\[
 0\leqslant h_m(f)-h_m^u(f)
 \leqslant\eps+\frac{4D\eps}{\alpha_0}\longrightarrow0
 \qquad(\eps\downarrow0).
\]
For each \(\eps\), the \(n\to\infty\) limit is taken with
all auxiliary choices fixed; the exponent \(\alpha_0>0\) is
independent of \(\eps\).
\end{proof}

\subsection{The full entropy formula}\label{subsec:full-proof}

\begin{proof}[Proof of \cref{cor:full}]
If \(u=0\), the assertion follows from Ruelle's inequality.
Assume \(u\geqslant1\). The first-level argument of
\cite[Section~5.3]{Tong}, following
\cite[Section~10.1]{LY2}, uses only domination of
\(E^{u,1}\oplus E^{\widehat{(1)}}\) and the common exponent
\(\lambda_1\) on \(E^{u,1}\). It therefore gives
\[
 \ldel_1=\udel_1=\delta_1=\frac{h_1}{\lambda_1},\qquad
 0\leqslant\delta_1\leqslant d_1
\]
without a multiplicity restriction.

  We now use the induction of \cite[Section~5.3]{Tong}, replacing
its transverse formula by \cref{thm:main}. Each required
three-way splitting is a coarsening of \eqref{eq:full-splitting}.
If the conditional measures at level \(j-1\) are exact dimensional, then
\cref{thm:main} gives
\[
 \ldel_j=\udel_j
 =\delta_{j-1}+\frac{h_j-h_{j-1}}{\lambda_j},
 \qquad 2\leqslant j\leqslant u.
\]
Thus the conditional measures are exact dimensional at every level.
Setting \(\gamma_j=\delta_j-\delta_{j-1}\), we obtain
\[
 0\leqslant\gamma_j\leqslant d_j,\qquad
 h_j=\sum_{k=1}^j\lambda_k\gamma_k.
\]
Finally, \(E^u\oplus E^{cs}\) is also a coarsening of
\eqref{eq:full-splitting}. Hence \cref{thm:unstable} gives
\(h_m(f)=h_m^u(f)=h_u\), proving \eqref{eq:full-formula}.
\end{proof}


\begin{thebibliography}{99}
\setlength{\baselineskip}{11.7pt}

\bibitem{ABC}
F.~Abdenur, C.~Bonatti, and S.~Crovisier, \emph{Nonuniform hyperbolicity for
  {$C^1$}-generic diffeomorphisms}, Israel J. Math. \textbf{183} (2011), no.~1,
  1--60.

\bibitem{AV}
A.~Avila and M.~Viana, \emph{Extremal {Lyapunov} exponents: an invariance
  principle and applications}, Invent. Math. \textbf{181} (2010), no.~1,
  115--178.

\bibitem{BPS}
L.~Barreira, Ya.~B. Pesin, and J.~Schmeling, \emph{Dimension and product
  structure of hyperbolic measures}, Ann. of Math. (2) \textbf{149} (1999),
  no.~3, 755--783.

\bibitem{BOLeaf}
S.~Ben~Ovadia, \emph{Invariant family of leaf measures and the
  {Ledrappier--Young} property for hyperbolic equilibrium states}, Ergodic
  Theory Dynam. Systems \textbf{43} (2023), no.~11, 3603--3635.

\bibitem{BOTubular}
S.~Ben~Ovadia, \emph{Tubular dimension: leaf-wise asymptotic local product structure,
  and entropy and volume growth}, Comm. Math. Phys. \textbf{407} (2026), no.~8,
  Paper No.~179.

\bibitem{BOHR}
S.~Ben~Ovadia and F.~Rodriguez-Hertz, \emph{Neutralized local entropy and
  dimension bounds for invariant measures}, Int. Math. Res. Not. IMRN
  \textbf{2024} (2024), no.~11, 9469--9481.

\bibitem{BCS}
C.~Bonatti, S.~Crovisier, and K.~Shinohara, \emph{The {$C^{1+\alpha}$}
  hypothesis in {Pesin} theory revisited}, J. Mod. Dyn. \textbf{7} (2013),
  no.~4, 605--618.

\bibitem{Brown}
A.~Brown, \emph{Smoothness of stable holonomies inside center-stable
  manifolds}, Ergodic Theory Dynam. Systems \textbf{42} (2022), no.~12,
  3593--3618.

\bibitem{BW}
K.~Burns and A.~Wilkinson, \emph{On the ergodicity of partially hyperbolic
  systems}, Ann. of Math. (2) \textbf{171} (2010), no.~1, 451--489.

\bibitem{CWZ}
Y.~Cao, J.~Wang, and Y.~Zhao, \emph{Dimension approximation in smooth dynamical
  systems}, Ergodic Theory Dynam. Systems \textbf{44} (2024), no.~2, 383--407.

\bibitem{CCE}
E.~Catsigeras, M.~Cerminara, and H.~Enrich, \emph{The {Pesin} entropy formula
  for {$C^1$} diffeomorphisms with dominated splitting}, Ergodic Theory Dynam.
  Systems \textbf{35} (2015), no.~3, 737--761.

\bibitem{CP}
S.~Crovisier and M.~Poletti, \emph{Invariance principle and non-compact center
  foliations}, 2023, Preprint, arXiv:2210.14989, version 2.

\bibitem{DGT}
N.~Ding, S.~Gan, and Y.~Tong, \emph{{SRB} measures for
  {$C^{1+\mathrm{Dini}}$} diffeomorphisms}, 2026, Preprint,
  arXiv:2607.13530, version 1.

\bibitem{DQ}
C.~Dong and Q.~Qiao, \emph{On {$r$}-neutralized entropy: entropy formula and
  existence of measures attaining the supremum}, Comm. Math. Phys. \textbf{406}
  (2025), no.~4, Paper No.~74.

\bibitem{ER}
J.-P. Eckmann and D.~Ruelle, \emph{Ergodic theory of chaos and strange
  attractors}, Rev. Mod. Phys. \textbf{57} (1985), no.~3, 617--656.

\bibitem{Epstein}
D.~B.~A. Epstein, \emph{Foliations with all leaves compact},
  Ann. Inst. Fourier (Grenoble) \textbf{26} (1976), no.~1, 265--282.

\bibitem{GTY}
S.~Gan, Y.~Tong, and J.~Yang, \emph{{Ledrappier--Young} entropy formula for
  {$C^1$} diffeomorphisms with dominated splitting. {Part 1}: Unstable entropy
  formula and invariance principle}, 2025, Preprint, arXiv:2509.16981, version
  3.

\bibitem{HPS}
M.~W. Hirsch, C.~C. Pugh, and M.~Shub, \emph{Invariant manifolds}, Lecture
  Notes in Mathematics, vol. 583, Springer-Verlag, Berlin--Heidelberg, 1977.

\bibitem{HHW}
H.~Hu, Y.~Hua, and W.~Wu, \emph{Unstable entropies and variational principle
  for partially hyperbolic diffeomorphisms}, Adv. Math. \textbf{321} (2017),
  31--68.

\bibitem{KH}
A.~Katok and B.~Hasselblatt, \emph{Introduction to the modern theory of
  dynamical systems}, Encyclopedia of Mathematics and its Applications,
  vol.~54, Cambridge University Press, Cambridge, 1995.

\bibitem{LedrappierCocycles}
F.~Ledrappier, \emph{Positivity of the exponent for stationary sequences
  of matrices}, in \emph{Lyapunov exponents (Bremen, 1984)}, Lecture Notes
  in Mathematics, vol.~1186, Springer-Verlag, Berlin, 1986, pp.~56--73.

\bibitem{LedrappierSinai}
F.~Ledrappier, \emph{Propri{\'e}t{\'e}s ergodiques des mesures de
  {Sina{\"\i}}}, Publ. Math. Inst. Hautes \'Etudes Sci. \textbf{59} (1984),
  163--188.

\bibitem{LS}
F.~Ledrappier and J.-M. Strelcyn, \emph{A proof of the estimation from below in
  {Pesin}'s entropy formula}, Ergodic Theory Dynam. Systems \textbf{2} (1982),
  no.~2, 203--219.

\bibitem{LW}
F.~Ledrappier and P.~Walters, \emph{A relativised variational principle
  for continuous transformations}, J. London Math. Soc. (2)
  \textbf{16} (1977), no.~3, 568--576.

\bibitem{LX}
F.~Ledrappier and J.-S. Xie, \emph{Vanishing transverse entropy in smooth
  ergodic theory}, Ergodic Theory Dynam. Systems \textbf{31} (2011), no.~4,
  1229--1235.

\bibitem{LY1}
F.~Ledrappier and L.-S. Young, \emph{The metric entropy of diffeomorphisms.
  {Part I}. {Characterization} of measures satisfying {Pesin}'s entropy
  formula}, Ann. of Math. (2) \textbf{122} (1985), no.~3, 509--539.

\bibitem{LY2}
F.~Ledrappier and L.-S. Young, \emph{The metric entropy of diffeomorphisms. {Part II}. {Relations}
  between entropy, exponents and dimension}, Ann. of Math. (2) \textbf{122}
  (1985), no.~3, 540--574.

\bibitem{LTTY}
G.~Liao, H.~Tao, Y.~Tong, and J.~Yang, \emph{Continuity properties of
  partial entropy}, 2026, Preprint, arXiv:2605.13273, version 1.

\bibitem{LVY}
G.~Liao, M.~Viana, and J.~Yang, \emph{The entropy conjecture for
  diffeomorphisms away from tangencies}, J. Eur. Math. Soc. (JEMS) \textbf{15}
  (2013), no.~6, 2043--2060.

\bibitem{Liu03}
P.-D. Liu, \emph{{Ruelle} inequality relating entropy, folding entropy and
  negative {Lyapunov} exponents}, Comm. Math. Phys. \textbf{240} (2003), no.~3,
  531--538.

\bibitem{Liu08}
P.-D. Liu, \emph{Invariant measures satisfying an equality relating entropy,
  folding entropy and negative {Lyapunov} exponents}, Comm. Math. Phys.
  \textbf{284} (2008), no.~2, 391--406.

\bibitem{LiuXie}
P.-D. Liu and J.-S. Xie, \emph{Dimension of hyperbolic measures of random
  diffeomorphisms}, Trans. Amer. Math. Soc. \textbf{358} (2006), no.~9,
  3751--3780.

\bibitem{Mane}
R.~Ma{\~n}{\'e}, \emph{A proof of {Pesin}'s formula}, Ergodic Theory Dynam.
  Systems \textbf{1} (1981), no.~1, 95--102.

\bibitem{Oseledets}
V.~I. Oseledets, \emph{A multiplicative ergodic theorem. {Lyapunov}
  characteristic numbers for dynamical systems}, Trans. Moscow Math. Soc.
  \textbf{19} (1968), 197--231.

\bibitem{Pesin}
Ya.~B. Pesin, \emph{Characteristic {Lyapunov} exponents and smooth ergodic
  theory}, Russian Math. Surveys \textbf{32} (1977), no.~4, 55--114.

\bibitem{Pugh}
C.~C. Pugh, \emph{The {$C^{1+\alpha}$} hypothesis in {Pesin} theory}, Publ.
  Math. Inst. Hautes \'Etudes Sci. \textbf{59} (1984), 143--161.

\bibitem{QQX}
M.~Qian, M.-P. Qian, and J.-S. Xie, \emph{Entropy formula for random dynamical
  systems: relations between entropy, exponents and dimension}, Ergodic Theory
  Dynam. Systems \textbf{23} (2003), no.~6, 1907--1931.

\bibitem{QX}
M.~Qian and J.-S. Xie, \emph{Entropy formula for endomorphisms: relations
  between entropy, exponents and dimension}, Discrete Contin. Dyn. Syst.
  \textbf{21} (2008), no.~2, 367--392.

\bibitem{QZ}
M.~Qian and S.~Zhu, \emph{{SRB} measures and {Pesin}'s entropy formula for
  endomorphisms}, Trans. Amer. Math. Soc. \textbf{354} (2002), no.~4,
  1453--1471.

\bibitem{Rokhlin}
V.~A. Rokhlin, \emph{Lectures on the entropy theory of measure-preserving
  transformations}, Russian Math. Surveys \textbf{22} (1967), no.~5, 1--52.

\bibitem{Ruelle}
D.~Ruelle, \emph{An inequality for the entropy of differentiable maps}, Bol.
  Soc. Brasil. Mat. \textbf{9} (1978), no.~1, 83--87.

\bibitem{Saghin}
R.~Saghin, \emph{On invariant holonomies between centers}, Ergodic Theory
  Dynam. Systems \textbf{45} (2025), no.~1, 274--293.

\bibitem{Shu09}
L.~Shu, \emph{The metric entropy of endomorphisms}, Comm. Math. Phys.
  \textbf{291} (2009), no.~2, 491--512.

\bibitem{Shu10}
L.~Shu, \emph{Dimension theory for invariant measures of endomorphisms}, Comm.
  Math. Phys. \textbf{298} (2010), no.~1, 65--99.

\bibitem{SunTian}
W.~Sun and X.~Tian, \emph{Dominated splitting and {Pesin}'s entropy formula},
  Discrete Contin. Dyn. Syst. \textbf{32} (2012), no.~4, 1421--1434.

\bibitem{Tahzibi}
A.~Tahzibi, \emph{{$C^1$}-generic {Pesin}'s entropy formula}, C. R. Math. Acad.
  Sci. Paris \textbf{335} (2002), no.~12, 1057--1062.

\bibitem{TY}
A.~Tahzibi and J.~Yang, \emph{Invariance principle and rigidity of high entropy
  measures}, Trans. Amer. Math. Soc. \textbf{371} (2019), no.~2, 1231--1251.

\bibitem{Tong}
Y.~Tong, \emph{{Ledrappier--Young} entropy formula for {$C^1$} diffeomorphisms
  with dominated splitting. {Part 2}: Entropy formulas and measure dimension},
  2026, Preprint, arXiv:2606.13384, version 1.

\bibitem{UVYYTransverse}
R.~Ures, M.~Viana, F.~Yang, and J.~Yang, \emph{Maximal transverse measures of
  expanding foliations}, Comm. Math. Phys. \textbf{405} (2024), no.~5, Paper
  No.~121.

\bibitem{UVYYu}
R.~Ures, M.~Viana, F.~Yang, and J.~Yang, \emph{Thermodynamical {$u$}-formalism {I}: measures of maximal
  {$u$}-entropy for maps that factor over {Anosov}}, Ergodic Theory Dynam.
  Systems \textbf{44} (2024), no.~1, 290--333.

\bibitem{WC}
J.~Wang and Y.~Cao, \emph{The {Hausdorff} dimension estimation for an ergodic
  hyperbolic measure of {$C^1$}-diffeomorphism}, Proc. Amer. Math. Soc.
  \textbf{144} (2016), no.~1, 119--128.

\bibitem{WWZ}
X.~Wang, L.~Wang, and Y.~Zhu, \emph{Formula of entropy along unstable
  foliations for {$C^1$} diffeomorphisms with dominated splitting}, Discrete
  Contin. Dyn. Syst. \textbf{38} (2018), no.~4, 2125--2140.

\bibitem{Yang}
J.~Yang, \emph{Entropy along expanding foliations}, Adv. Math. \textbf{389}
  (2021), Paper No.~107893.

\bibitem{Young}
L.-S. Young, \emph{Dimension, entropy and {Lyapunov} exponents}, Ergodic Theory
  Dynam. Systems \textbf{2} (1982), no.~1, 109--124.

\end{thebibliography}
\end{document}